\documentclass[times,sort&compress,3p,preprint]{elsarticle}
\usepackage{latexsym,epsfig,amssymb,amsmath,amsthm,url,bbm,enumerate,float,dsfont}

\usepackage{verbatim}
\usepackage[usenames,dvipsnames]{color}
\usepackage{longtable}

\usepackage{cleveref}

\newtheorem{Theorem}{Theorem}[section]
\crefname{Theorem}{Theorem}{Theorems}
\Crefname{Theorem}{Theorem}{Theorems}
\newtheorem{Corollary}[Theorem]{Corollary}
\crefname{Corollary}{Corollary}{Corollaries}
\Crefname{Corollary}{Corollary}{Corollaries}
\newtheorem{Proposition}[Theorem]{Proposition}
\crefname{Proposition}{Proposition}{Propositions}
\Crefname{Proposition}{Proposition}{Propositions}
\newtheorem{Lemma}[Theorem]{Lemma}
\crefname{Lemma}{Lemma}{Lemmata}
\Crefname{Lemma}{Lemma}{Lemmata}

\crefname{Assumption}{Assumption}{Assumptions}

\Crefname{Assumption}{Assumption}{Assumptions}

\usepackage{chngcntr}
\counterwithin{figure}{section}

\newtheoremstyle{note}%
{3pt}%
{3pt}%
{}%
{}%
{\bf}%
{}%
{.5em}%
{\thmname{#1}\thmnumber{ #2} \thmnote{\sc{(#3)}}}

\theoremstyle{note}
\newtheorem{Definition}[Theorem]{Definition}
\crefname{Definition}{Definition}{Definitions}
\Crefname{Definition}{Definition}{Definitions}
\newtheorem{Remark}[Theorem]{Remark}
\crefname{Remark}{Remark}{Remarks}
\Crefname{Remark}{Remark}{Remarks}

\newtheorem{Example}[Theorem]{Example}
\crefname{Example}{Example}{Examples}
\Crefname{Example}{Example}{Examples}

\newtheorem{modelletter}{Model}

\def\la{\leftarrow}

\def\C{\mathbb{C}}

\def\E{\mathbb{E}}

\def\M{\mathbb{M}}

\def\R{\mathbb{R}}

\def\bV{\boldsymbol V}

\def\bY{\boldsymbol Y}
\def\bZ{\boldsymbol Z}
\def\bR{\boldsymbol R}

\def\bx{\boldsymbol x}
\def\by{\boldsymbol y}

\def\RV{\mathcal{RV}}
\def\MRV{\mathcal{MRV}}
\def\HRV{\mathcal{HRV}}

\def\MES{\text{MES}}
\def\VaR{\text{VaR}}
\def\MME{\text{MME}}

\definecolor{darkred}{RGB}{139,0,0}
\definecolor{darkgreen}{RGB}{0,139,0}

\newcommand{\ov}{\overline}

\newcommand{\wh}{\widehat}
\newcommand{\wt}{\widetilde}
\newcommand{\1}{{\mathds{1}}}

\begin{document}


\begin{frontmatter}

\title{Hidden regular variation, copula models, and \\the limit behavior of conditional excess risk measures}

\author[bd]{Bikramjit Das}\ead{bikram@sutd.edu.sg}
\author[vf]{Vicky Fasen-Hartmann}\ead{vicky.fasen@kit.edu}

\address[bd]{Singapore University of Technology and Design, 8 Somapah Road, Singapore 487372}
\address[vf]{Karlsruhe Institute of Technology, Englerstrasse 2, 76131 Karlsruhe, Germany}

\fntext[fn1]{B. Das gratefully acknowledges support from MOE Tier 2 grant MOE2017-T2-2-161}


%
%
%
%
%
%
%
%
%
%

\begin{abstract}

Risk measures like  \emph{Marginal Expected Shortfall} and  \emph{Marginal Mean Excess} quantify
conditional risk and in particular, aid in the understanding of systemic risk. In many such scenarios, models
exhibiting heavy tails in the margins and asymptotic tail independence in the joint behavior play a fundamental role.
The notion of \emph{hidden regular variation} has the advantage that  it models both properties: asymptotic tail independence as well as
heavy tails. An alternative approach to addressing these features is via copulas.  First, we
elicit connections between hidden regular variation and the behavior of tail copula parameters extending  previous works in this area.
Then we study the asymptotic behavior of the aforementioned conditional excess risk measures; first under hidden regular variation
and then under restrictions  on the tail copula parameters, not necessarily assuming hidden regular variation. We provide a broad variety of
examples of models admitting heavy tails and asymptotic tail independence
 along with  hidden regular variation and with the appropriate limit behavior for the risk measures of interest.

 \end{abstract}

\begin{keyword}
{asymptotic tail independence} \sep
{copula models} \sep
{expected shortfall}  \sep
{heavy-tail}  \sep
{hidden regular variation}  \sep
{mean excess}  \sep
{multivariate regular variation}  \sep
{systemic risk}
\end{keyword}

%

\end{frontmatter}


\section{Introduction}\label{sec:intro}

In practice, one often encounters risk factors which are heavy tailed in nature; which means values further away from the
 mean have a relatively high probability of occurring than for example for exponentially tailed distributions like normal or exponential; see
 \cite{anderson:meerschaert:1998,crovella:bestavros:taqqu:1998,embrechts:kluppelberg:mikosch:1997,McNeil:Frey:Embrechts,smith:2003} for details. The joint behavior of such multi-dimensional heavy-tailed
 random variables are often studied using the notion of \emph{multivariate regular variation} (MRV); see
 \cite{bingham:goldie:teugels:1989,Resnick:2007}. For our paper we resort to this approach.

In certain scenarios, such multivariate models exhibit a phenomenon called \emph{asymptotic tail independence}; see
\citet{Poon:Rockinger:Tawn} for empirical evidence of asymptotic tail independence and heavy tails in five major
stock market indices.
In this paper we restrict to study non-negative variables and hence, our interest is in the  upper tail dependence between different variables. To that end we interchangeably use the names \emph{asymptotic tail independence}
and \emph{asymptotic upper tail independence} for  the same notion. For any random variable $Z$, let $F_{Z}$ denote its probability distribution function.
For a bivariate random vector $\bZ=(Z_1,Z_2)$ we can  define \emph{asymptotic tail independence (in the upper tails)} as
\begin{eqnarray} \label{def:asyind}
    \lim_{p\to 0}\Pr(Z_1>F_{Z_1}^{\leftarrow}(1-p)|Z_2>F_{Z_2}^{\leftarrow}(1-p))=0
\end{eqnarray}
where $F_{Z_i}^{\leftarrow}(1-p)=\inf\{x\in\R:\,F_{Z_i}(x)\geq 1-p\}$ is the generalized inverse of $F_{Z_i}$.
Hence, the presence of asymptotic tail independence among $Z_{1}$ and $Z_{2}$  implies that it is highly unlikely for the two random variables to take extreme values together. This phenomenon aptly noted by
\citet{sibuya:1960} more than half a century back, especially in the context of the very popular and useful bivariate normal distribution has been a source of intrigue and further research by many.

 A key notion in defining multivariate regularly varying distributions under the assumption of asymptotic tail independence is \emph{hidden regular variation}; see
 \cite{resnick:2002,das:mitra:resnick:2013,Ledford:Tawn}. The family of hidden regular varying distributions is a semi-parametric subfamily of the
 full-family of distributions possessing multivariate regular variation.
 In the past few years, researchers have explored the connection between hidden regular variation and copula models via \emph{tail dependence functions}
 \citep{Zhu:Li:2012,Joe:Li:2011,Hua:Joe:Li:2014,Li:Hua:2015}, and \emph{weak tail dependence functions} \citep{tankov:2016a}. In this paper we explore this
 further and provide insight into the connections between hidden regular variation and copula models.
 Furthermore, we construct multivariate models exhibiting hidden regular variation  as used in a systemic risk context  using both
 additive models as well as  copula models.

   In the univariate context a popular risk measure is the   Expected Shortfall (ES), also known as Conditional Tail Expectation (CTE)
  and Tail-Value-at-Risk (TVaR). It is widely used in practice and also incorporated in the regulatory frameworks of Basel III for banks  and
   Solvency II for insurances.
 In systems with more than one variable, it is often of interest to judge the risk behavior of one component given a high risk or stress in the others.
Hence, we resort to computing  conditional excess risk measures  which include the \emph{Marginal Expected Shortfall} (MES) and the \emph{Marginal Mean Excess} (MME).
The MES is well-known in many contexts, and has been especially proposed for measuring systemic risk \citep{acharya:petersen:philippon:richardson:2010,brownlees:engle:2017,zhou:2010}.

In order to define the above risk measures, recall that for a random variable $Z$ and $p \in (0,1)$ the Value-at-Risk (VaR) at level $p$ is the quantile function
$$  \VaR_p(Z)= \inf\{x\in\R: \Pr(Z > x) \le 1- p \} = \inf\{x\in\R: \Pr(Z \le x) \ge p \}.$$
 Note that  smaller values of $p$ lead to higher values of $\VaR_{{p}}$.
Now suppose  that $\bZ=(Z_1, Z_2)\in\left[0,\infty\right)^2$
  denotes the risk exposure of a financial institution, and $Z_1$ and $Z_2$ are the marginal risks of two risk factors.
We intend to study the expected behavior of one risk, given that the other risk is high; and the following two measures achieve this.
For $\E|Z_1|<\infty$ the MES at level $p \in (0,1)$ is defined as
  \begin{align}\label{def:MES}
  \MES(p) = \E\{Z_1  |\, Z_2>\VaR_{1-p}(Z_2)\}, \quad\quad\quad\quad\quad\quad\quad\;\;\;
  \end{align}
and the  MME at level $p \in (0,1)$ is defined as
  \begin{align}\label{def:MME}
  \MME(p) = \E \{ (Z_1-\VaR_{1-p}(Z_2) )_+  |\,Z_2>\VaR_{1-p}(Z_2)\}.
  \end{align}
The measure MES represents the expected shortfall of one risk given that the other risk is higher than its Value-at risk at level $1-p$, whereas
the measure MME represents the expected excess of risk $Z_{1}$ over the Value-at-Risk of $Z_2$ at level $1-p$ given that the value of $Z_{2}$
is already greater than the same Value-at-Risk; see \citet{DasFasen2017} for details where the measure MME is also defined. Note that the measure $\MES(p)$ is equal
to the expected shortfall (ES)  if $Z_{1}\equiv Z_{2}$.
In this context we also consider a few other extensions of ES:
  \begin{eqnarray} \label{I1}
     &&\MES^{+}(p)=\E \{ Z_1  |\,Z_1+Z_2>\VaR_{1-p}(Z_1+Z_2)\}, \nonumber\\
     &&\MES^{\min}(p)=\E \{ Z_1  |\,\min(Z_1,Z_2)>\VaR_{1-p}(\min(Z_1,Z_2))\},\quad  \\
     &&\MES^{\max}(p)=\E \{ Z_1  |\,\max(Z_1,Z_2)>\VaR_{1-p}(\max(Z_1,Z_2))\}, \nonumber
  \end{eqnarray}
  see \citet{Cai:Li:2005,Cousin:Bernardino:2013}.
  The idea is that $Z_1+Z_2$ is the aggregate risk of the institution,
  and $\min(Z_1,Z_2)$ and $\max(Z_1,Z_2)$ are the extremal risks. The risk measure $\MES^{+}$ is associated with the Euler
  allocation rule; see \citet[Section 6.3]{McNeil:Frey:Embrechts}.
Further interpretations of these risk measures in finance and
  insurance are elaborated in \citet{Cai:Li:2005}.

The asymptotic tail behavior of MES was discussed under the assumption of  regular variation and asymptotic tail dependence  in
\citet{cai:einmahl:dehaan:zhou:2015,Hua:Joe:2011b}, and under  asymptotic tail independence in \citet{DasFasen2017,CaiMusta2017}.
Such risk metrics have also been studied in a time-series context in \citet{KulikSoulier2015}, and under a copula-framework in \citet{Hua:Joe:2014b}.
In the asymptotic tail-dependent case it has been observed that
\[\MES(p)\sim \MME(p) \sim \VaR _{1-p}(Z_1) \quad \text{as } p\downarrow 0,\] independent of the structure of the dependence. Interestingly, in the
asymptotically tail-independent case these risk measures may have  different rates of convergence or  even converge to a constant, e.g., for independent
random variables $Z_1,Z_2$ we have $\MES(p)=\E(Z_1)$. The asymptotic behavior of MME and MES for a hidden regularly varying
random vector $\bZ=(Z_1,Z_2)$ has been investigated in \citet{DasFasen2017}, furthermore, consistent estimators for these risk measures based on methods from extreme value theory have been proposed; for asymptotic normality of MES see
\citet{CaiMusta2017}.
On the other hand, the asymptotic behavior of the risk measures as defined in \eqref{I1} are not particularly well-studied, especially in the context of heavy-tailed asymptotically tail independent risks.
Explicit formulas for $\MES^{+}$ are given in
\citet{Chiragiev:Landsman:2007} for multivariate Pareto distributions, in \citet{Cai:Li:2005} for multivariate phase-type distributions,
in \citet{Barges:Cossette:Marceau}  for light tailed risks with Farlie-Gumbel-Morgenstern copula and in
\citet{Landsman:Valdez} for elliptical distributions.

  In this paper we extend  the work of \citet{DasFasen2017}  for the asymptotic tail behavior of MME and MES and set it in a more general framework.
  On the one hand, we derive the limit behavior of MES and MME for a variety of  hidden regularly varying models using the results from
  \cite{DasFasen2017}. Asymptotic limits are also obtained for the risk measures given in \eqref{I1} pursuing a similar approach. On the other hand, we formulate the asymptotic behavior of these conditional risk measures under very general assumptions
  on the copula tail parameters
  showing that the limit behavior for $p\downarrow 0$ of $\MES(p)$ and $\MME(p)$ are similar to those in the case of hidden regular variation.  In particular, we show that the asymptotic behavior of
  $\MES(p)$ depends only on the tail of $Z_1$ and the tail copula; neither the presence of hidden regular variation nor the tail of $Z_2$ have any influence
  on the limit.  We compare  the conclusions for $\MES(p)$ with those of  \citet{CaiMusta2017}. Finally,
  we provide several examples for models satisfying our assumptions.

    The paper is organized as follows.
  In Section \ref{sec:prelim} we give  a brief introduction into copulas and survival copulas along with notions of multivariate regular variation and hidden regular variation. Then,
  in Section \ref{sec:hrvcopmix}, we construct models exhibiting hidden regular variation  on the one hand, by  additive models as used in the systemic risk context
  and on the other hand, by copula models. The asymptotic behavior of MME, MES and the measures in \eqref{I1} under the models of \Cref{sec:hrvcopmix} are content of \Cref{Section 4.1}.
  Finally, in \Cref{Section 4.2}, we develop sufficient conditions on the copula tail parameters to obtain the asymptotic behavior
  for $\MME(p)$ and $\MES(p)$ as in \Cref{Section 4.1}  without assuming hidden regular variation.
     Conclusions are drawn  in \Cref{sec:concl} along with ideas for future directions of research in this domain.

  \vspace{10pt}



%

\section{Preliminaries}\label{sec:prelim}

In this section we discuss necessary tools and definitions for regular variation and copula theory which are  used in the subsequent sections. Details on regular variation defined using $\M$-convergence is available in
\cite{lindskog:resnick:roy:2014,hult:lindskog:2006a, das:mitra:resnick:2013} and
details on copulas and survival copulas can be found in \cite{Nelsen}.
Unless otherwise stated all random variables take non-negative values and we discuss copulas and regular variation in two-dimensions.
Moreover, for  vectors $\bx =(x_{1},x_{2})\in\R^2$, we denote by $\|\bx\|$ any suitable norm in $\R^{2}$.

\subsection{Copulas and survival copulas} \label{subsec:cop}

Copula theory is popularly used to separate out the marginal behavior of random variables from their dependence structure.  In two dimensions, a \textit{copula} is a  distribution function on $[0,1]^2$ with uniformly distributed margins.
Using Sklar's theorem
\citep{Nelsen}, we know that
for every bivariate distribution function $F$ with marginal distribution functions $F_i$ ($i=1,2$),
there exists a copula $C$ such that
\begin{eqnarray} \label{Sklar}
    F(x,y )= C(F_1(x),F_2(y)) \quad \text{for } (x,y)\in\R^2.
\end{eqnarray}
If $F_1$, $F_2$ are continuous then $C$ is uniquely defined by
\begin{eqnarray*}
    C(u,v)=F(F_1^{\leftarrow}(u),F_2^{\leftarrow}(v)) \quad  \text{for } (u,v)\in [0,1]^2.
\end{eqnarray*}
Denoting the survival or tail distribution of $F_i$ by $\ov F_i (x) := 1- F_i(x)$, a version of \eqref{Sklar}
applies also to the joint survival function $$\ov F(x,y):=\Pr(Z_1>x,Z_2>y)=\ov F_1(x)+\ov F_2(y)-(1-F(x,y))$$
of the bivariate random vector $\bZ=(Z_1,Z_2)$ with distribution function $F$ and margins $F_1, F_2$. In this case, there
exists again a copula $\wh C$, the \textit{survival copula}, such that
\begin{eqnarray*}
    \ov F(x,y) =\wh C(\ov F_{1}(x),\ov F_{2}(y)) \quad \text{for } (x,y)\in\R^2.
\end{eqnarray*}
Moreover, in the bivariate case $C$ and $\wh C$ are related by
\begin{eqnarray*}
    \wh C(u,v)=u+v-1+ C(1-u,1-v)\quad \text{for } (u,v)\in[0,1]^2.
\end{eqnarray*}
Since we are interested in the dependence (as well as independence) in the upper tails, the behavior of the survival copula $\wh C(u,v)$ for $u,v$ close to 1 will be of
 significance in this paper. The relationship between the survival copula and \emph{hidden regular variation} is discussed further in Section  \ref{subsec:copsurv}.

\subsection{Regular variation and hidden regular variation} \label{subsec:rvhrv}

Recall that a  measurable function $f:(0,\infty)\to(0,\infty)$ is
 \textit{regularly varying} at $\infty$ with index $\rho \in \R$ if \linebreak $\lim_{t\to\infty} {f(tx)}/{f(t)}=x^{\rho}$ for any $x>0$
and we write $f\in\RV_\rho$; in contrast, we say $f$ is regularly varying at $0$ with index $\rho$ if $\lim_{t\to 0}f(tx)/f(t)=x^{\rho}$ for any $x>0$.
 In this paper, unless otherwise specified, regular variation means regular variation at infinity.
A random variable $Z$ with distribution function $F_Z$ has a regularly varying tail if $\overline{F}_Z=1-F_Z \in \RV_{-\alpha}$ for some $\alpha\ge 0$. We often write
 $Z\in \RV_{-\alpha}$ by abuse of notation.
We define \emph{multivariate regular variation} using $\M$-convergence; see \citet{lindskog:resnick:roy:2014}. All notions are restricted to $\left[0,\infty\right)^2$ and
its' subspaces. Suppose $\C_0 \subset \C \subset \left[0,\infty\right)^2$ where
$\C_0$ and $\C$ are closed cones containing $\{(0,0)\} \in\R^2$.
Denote by  $\M(\C\setminus\C_0)$  the class of Borel measures on $\C\setminus\C_0$ which are
finite on subsets bounded away from  $\mathbb{C}_0$. For functions $f:[0,\infty)^{2}\to\R$, denote by $\mu(f):=\int f \,\mathrm d \mu.$
Then $\mu_n \stackrel{\M}{\to}\mu$ in $\M(\C \setminus \C_0)$ if $\mu_n(f)\to\mu(f)$ for all continuous
and bounded functions $f: \C \setminus \C_0 \to \R$ whose supports are bounded away from $\C_0$.

\begin{Definition}[Multivariate regular variation]\label{def:mrv}
 A random vector ${\bZ } =(Z_1,Z_2) \in \C$ is \emph{(multivariate) regularly
varying} on $\C \setminus \C_0$,
 if there {exist} a function $b(t) \uparrow \infty$ and a non-zero measure {$\nu(\cdot)\in \M(\C \setminus \C_0)$} such that as $t \to \infty$,
 \begin{equation}\label{eqn:ccc}
 \nu_{t}(\cdot):=t\,\Pr\left({{\bZ }}/{b(t)} \in \; \cdot \; \right) \stackrel{\M}{\to} \nu(\cdot) \hskip 0.5 cm \text{in {$\M(\C \setminus \C_0)$.}}
\end{equation}
The limit measure has the homogeneity property: $\nu(cA)=c^{-\alpha}\nu(A)$ for some $\alpha>0$. We write $\bZ\in \MRV(\alpha, b,\nu, \C \setminus \C_0)$ and
sometimes write MRV for multivariate regular variation.
\end{Definition}
Classically, MRV is defined on the space $\E=\left[0,\infty\right)^2\setminus\{(0,0)\} = \C \setminus \C_0$ where $\C=\left[0,\infty\right)^2$ and $\C_0=\{(0,0)\}$.
Sometimes it is possible and perhaps necessary to define further regular variation on subspaces of $\E$, since the limit measure $\nu$ as obtained in
\eqref{eqn:ccc}  concentrates on a proper subspace of $\E$. The most likely way this happens is through   \emph{asymptotic tail independence} of a random  vector
as defined in \eqref{def:asyind}. The property can be nicely described by using the survival copula function (see \citet{McNeil:Frey:Embrechts}).
Note that,  we say that $Z_{1}$ and $Z_{2}$ are \emph{tail-equivalent}
\linebreak if $\lim\limits_{t\to\infty} \Pr(Z_{1}>t)/\Pr(Z_{2}>t)$ exists in $(0,\infty)$.

\begin{Lemma}\label{lem:basicconnect}
Suppose $\bZ=(Z_{1}, Z_{2}) \in \MRV(\alpha,b,\nu,\E)$ having continuous marginals and survival copula $\wh C$. Consider the following statements.
\begin{enumerate}[(a)]
\item $(Z_{1}, Z_{2})$ are asymptotically tail independent.
\item $\lim_{p\downarrow0} \wh C(p,p)/p =0.$
\item $\nu((0,\infty)\times(0,\infty))=0.$
\end{enumerate}
Then $(a) \iff (b)$, $(b) \implies (c)$, and if we assume that $Z_1$ and $Z_2$ are tail-equivalent then  $(c) \implies (a)$.
\end{Lemma}
 The lemma is easy to verify; more details can be found in \citet[Chapter 7]{reiss:1989}
and \citet[Proposition 5.27]{resnickbook:2008}. Independent random vectors are trivially asymptotically tail independent.
Note that fact (c) does not necessarily imply (a); this can be verified using the following counterexample: let $Z \in \RV_{-\alpha}$
for $\alpha>0$, then the random vector $(Z,Z^2)$ is multivariate regularly varying
with limit measure $\nu$ with $\nu((0,\infty)\times(0,\infty))=0$; but of course $(Z,Z^2)$ is asymptotically tail dependent.

  Consequently, if $(Z_{1}, Z_{2})$ are MRV,  asymptotically tail independent and the margins  are tail-equivalent  we would approximate  $ \Pr(Z_2>x|Z_1>x)\approx 0 $
 for large thresholds $x$ and conclude that risk contagion between $Z_1$ and $Z_2$ is absent. This conclusion may be naive
 and hence, the concept of \emph{hidden regular variation}  on $\E_0= \left[0,\infty\right)^2\setminus (\{0\}\times\left[0,\infty\right) \cup \left[0,\infty\right)\times\{0\}) =(0,\infty)^2$  was introduced in
 \citet{resnick:2002}.
{Note that we do not  assume that the marginal tails of $\bZ$ are necessarily tail-equivalent in order to define hidden regular variation, which is usually done in \cite{resnick:2002}.
\begin{Definition}[Hidden regular variation]\label{def:hrv}
A regularly varying random vector $\bZ$ on $\E$ possesses \emph{hidden regular variation} on $\E_0 = (0,\infty)^2$ with index $\alpha_0 (\ge \alpha >0)$ if there exist
 scaling functions $b(t)\in \RV_{1/\alpha}$ and $b_0(t)\in \RV_{1/\alpha_0}$
with $b(t)/b_0(t)\to\infty$ and limit measures $\nu,\nu_0$ such that
\begin{eqnarray*}
    \bZ \in \MRV({\alpha}, b, \nu, \E)\cap \MRV({\alpha_0}, b_0, \nu_0, \E_0).
\end{eqnarray*}
We write $\bZ\in \HRV({\alpha_0}, b_0, \nu_0)$ and sometimes write HRV for hidden regular variation.
\end{Definition}

For example, say $Z_1, Z_2$ are iid random variables with distribution function $F(x)=1-x^{-1}, x>1$.  Here $\bZ=(Z_1,Z_2)$ possesses
 MRV on $\E$, asymptotic tail independence and HRV on $\E_0$. Specifically, $$\bZ\in \MRV(\alpha=1, b(t)=t, \nu, \E) \cap \MRV(\alpha_{0}=2, b_{0}(t)=\sqrt{t}, \nu_0, \E_{0})$$ where for $x>0,y>0$, $$\nu([\{(0,0)\},(x,y)]^c) = \frac{1}{x} + \frac 1y \quad \text{ and } \quad \nu_0(\left[x,\infty\right)\times\left[y,\infty\right)) = \frac{1}{xy}.$$

A combination of \citet{resnick:2002,maulik:resnick:2005} and \citet[Lemma~1]{DasFasen2017} is the following.
\begin{Lemma} \label{Lemma 2.5}
    Let  $\bZ \in \MRV({\alpha}, b, \nu, \E)\cap \MRV({\alpha_0}, b_0, \nu_0,\E_{0})$. Then $$\bZ \in \MRV({\alpha}, b, \nu, \E)\cap \HRV({\alpha_0}, b_0, \nu_0,\E_{0})$$ iff $\bZ$
    is asymptotically
    tail independent.
\end{Lemma}

\section{Hidden regular variation in additive models and copula models}\label{sec:hrvcopmix}

In this section we investigate hidden regular variation properties of models that are generated using different methods; on one hand, we investigate additive models and
on the other hand, we investigate copula models. The models discussed in this section are then used in Section \ref{sec:MMEMES} to compute the asymptotic limits of
the conditional excess measures MME, MES and the measures in \eqref{I1}. Note that we concentrate on multivariate regularly varying models in a non-standard sense.
Hence,  $\bZ=(Z_{1}, Z_{2}) \in \MRV(\alpha)$ does not necessarily imply that both marginal variables have equivalent (or equal) tails; in fact,  both margins need not be regularly varying either.

\subsection{Hidden regular variation of additive models}\label{subsec:addmod}

Hidden regular variation properties of additive models (sometimes called mixture models) have been discussed in \citet{weller:cooley:2014} and \citet{das:resnick:2015} where the authors
concentrate on adding two standard regularly varying models to get an additive structure with hidden regular variation. The class of models we consider are more general in the
sense that the marginal tails of the additive components are not necessarily tail-equivalent. We establish the presence of hidden regular variation in these models under certain
regularity conditions. First, we state a result on hidden regular variation of independent regularly varying random variables.

\begin{Lemma} \label{Lemma1}
Suppose $\bY=(Y_1,Y_2)\in\left[0,\infty\right)^2$  where $Y_{1}$ and $Y_{2}$ are independent random variables with $\ov F_{Y_1}\in\RV_{-\alpha}$ and $\ov F_{Y_2}\in\RV_{-\alpha^*}$.
Then $\bY\in \MRV(\min(\alpha,\alpha^{*}))\cap\HRV(\alpha+\alpha^*)$.
\end{Lemma}
\begin{proof}
Either $\ov F_{Y_1}$ and $\ov F_{Y_2}$ are tail-equivalent or one tail is lighter tailed than the other.
Without loss of generality we assume that $Y_{2}$ has a tail lighter than or is tail-equivalent to that of $Y_{1}$, in particular $\alpha\le \alpha^{*}$ and
 \[\lim_{t\to\infty} \frac{\Pr(Y_{2}>t)}{\Pr(Y_{1}>t)} =C,\]
 where $C\in\left[0,\infty\right)$. Then for $x,y>0$ we have with $A_{x,y} = ([0,x]\times[0,y])^{c}$,
\begin{align*}
\Pr\left(\bY \in A_{x,y}\right)  = \Pr(Y_{1}>x) + \Pr(Y_{2}>y) - \Pr(Y_{1}>x)\Pr(Y_{2}>y).
\end{align*}
Hence,
\begin{align*}
\lim_{t\to\infty}\frac{\Pr\left(\bY \in tA_{x,y}\right)} {\Pr\left(\bY \in tA_{1,1}\right)} & = \frac{x^{-\alpha} + Cy^{-\alpha^{*}}}{1+C}
\end{align*}
 where either $\alpha=\alpha^{*}$ or  $C=0$. This implies $\bY \in \MRV(\alpha)$. We also have
\begin{eqnarray*}
    \Pr(Y_1>t,Y_2>t)=\Pr(Y_1>t)\Pr(Y_2>t)\in\RV_{-(\alpha+\alpha^*)},
\end{eqnarray*}
and  for $x,y>0$,
\begin{eqnarray*}
    \frac{\Pr\left(Y_1>x t,Y_2>yt\right)}{\Pr\left(Y_1>t,Y_2>t\right)}
        =\frac{\Pr\left(Y_1>xt\right)}{\Pr\left(Y_1>t\right)}\frac{\Pr\left(Y_2>yt\right)}{\Pr\left(Y_2>t\right)}
        \stackrel{t\to\infty}{\to}x^{-\alpha}y^{-\alpha^*}.
\end{eqnarray*}
Hence, $\bY$ is hidden regularly varying with index $\alpha+\alpha^*$.
\end{proof}
\begin{Remark}
For the additive models we consider next, one of the summands behaves like $\bY$ of Lemma \ref{Lemma1} and hence, the component to be added must have regular variation index smaller than  $\alpha+\alpha^{*}$ to provide a non-trivial generative set of limit models. 
\end{Remark}


\begin{modelletter} \label{additive:model}
Suppose $\bY=(Y_1,Y_2),\bV=(V_1,V_2),$ and $\bZ=(Z_1,Z_2)$ are
random  vectors in $\left[0,\infty\right)^2$ such that $\bZ=\bY+\bV$. Assume the following holds:
\begin{enumerate}
\item[(A1)] $\bY \in \MRV(\alpha,b,\nu,\E)$ and $Y_1,Y_2$ are independent.
\item[(A2)] \label{A2} $\bV \in \MRV(\alpha_0,b_0,\nu_0,\E)$ with $\alpha_{0} \ge \alpha$ and $\bV$ does not possess asymptotic tail independence. Moreover, 
\begin{eqnarray*}
    \lim_{t\to\infty}\frac{\Pr(\|\bV\|>t)}{\Pr(\|\bY\|>t)}=0.
\end{eqnarray*}
\item[(A3)] $\bY$ and $\bV$ are independent.
\item[(A4)]  Suppose $\ov F_{Y_2}\in\RV_{-\alpha^*}$ or $\ov F_{Y_2}(t)=o(t^{-\alpha^*})$ where
  $\alpha^*>\alpha_0-\alpha$.
\end{enumerate}
\end{modelletter}
Obviously, (A1) implies that either $\ov F_{Y_2}\in\RV_{-\alpha}$, or $\ov F_{Y_1}(t) \in \RV_{-\alpha}$ where $\lim_{t\to\infty} {\ov F_{Y_2}(t)}/{\ov F_{Y_1}(t)}=0$.
However, to use \Cref{Lemma1} we impose the additional
assumption (A4) with $\alpha^*>\alpha_0-\alpha$.
Note that $\bY$ and $\bV$ are independent of one another and both are multivariate regularly varying. The tail of
 $\|\bY\|$ is heavier than that of $\|\bV\|$, $\bV$ does not have asymptotic tail independence and hence, in turn no hidden regular variation on $\E_{0}=(0,\infty)^{2}$ can be defined for  $\bV$. The following theorem shows the existence of hidden regular variation in Model \ref{additive:model}. A version of the theorem is stated in
\citet[Theorem 3]{DasFasen2017} referring to the proof in the present paper.


\begin{Theorem}  \label{theorem:3.1}
Let  $\bZ=\bY+\bV$ be as in Model \ref{additive:model}.
Then  the following statements hold:
\begin{enumerate}[(a)]
    \item  $\bZ \in \MRV({\alpha}, b, \nu, \E)\cap \HRV({\alpha_0}, b_0, \nu_0, \E_0)$.
    \item If $Y_2\equiv0$, then   $\bZ^{+}=(Z_1+Z_2,Z_2)\in \MRV({\alpha}, b, \nu, \E)\cap \HRV({\alpha_0}, b_0, \nu_0^+, \E_0)$ with
    \begin{eqnarray*}
        \nu_0^{+}(A)&=&\nu_0(\{(v_1,v_2)\in\E_0:(v_1+v_2,v_2)\in A\}) \quad \text{ for }A\in\mathcal{B}(\E_0).
    \end{eqnarray*}
    \item If  $\liminf\limits_{t\to\infty}\Pr(Y_1>t)/\Pr(Y_2>t)>0$, then $\bZ^{\min}=(Z_1,\min(Z_1,Z_2))\in \MRV({\alpha}, b, \nu^{\min}, \E)
    \cap \HRV({\alpha_0}, b_0, \nu_0^{\min}, \E_0)$ with
        \begin{eqnarray*}
        \nu^{\min}(A)&=&\nu(\{(y_1,0)\in\E:(y_1,0)\in A\}) \quad \text{ for }A\in\mathcal{B}(\E),\\
        \nu_0^{\min}(A)&=&\nu_0(\{(v_1,v_2)\in\E_0:(v_1,\min(v_1,v_2))\in A\}) \quad \text{ for }A\in\mathcal{B}(\E_0).
    \end{eqnarray*}
    \item If $Y_2\equiv0$, then $\bZ^{\max}=(\max(Z_1,Z_2),Z_2)\in \MRV({\alpha}, b, \nu, \
    \E)\cap \HRV({\alpha_0}, b_0, \nu_0^{\max}, \E_0)$ with
    \begin{eqnarray*}
        \nu_0^{\max}(A)&=&\nu_0(\{(v_1,v_2)\in\E_0:(\max(v_1,v_2),v_2)\in A\}) \quad \text{ for }A\in\mathcal{B}(\E_0).
    \end{eqnarray*}
    \end{enumerate}
\end{Theorem}
Since our aim is often to find and compare systemic risk in the presence of risk factors pertaining to two institutions, Theorem \ref{theorem:3.1} addresses
different kinds of measures for systemic risk in this context. If risk is additive and we compare risk of one with that of the portfolio of the system we refer to (b), if
risk is measured in terms of both institutions being at risk we refer to (c), and in case systemic risk pertains to any of the institutions being in risk, we refer to part (d).
Thus, a gamut of systemic risk measurement can be addressed under Model \ref{additive:model}.

\begin{proof}[Proof of \Cref{theorem:3.1}] 
$\mbox{}$
\begin{enumerate}
\item[(a)] \textbf{Step 1.} We get $\bZ \in \MRV({\alpha}, b, \nu, \E)$ using \citet[Lemma~3.12]{jessen:mikosch}.\\
\textbf{Step 2.}  The proof of $\bZ\in\MRV(\alpha_0,b_0,\nu_0,\E_0)$ is analogous to the proof of Proposition~3.2 in \citet{das:resnick:2015}  and skipped here.

\item[(b)] Define $\bY^+=(Y_1,0)$, $\bV^+=(V_1+V_2,V_2)$ and write $\bZ^+=\bY^++\bV^+$. Now, we can check that $\bY^+ \in \MRV(\alpha,b,\nu,\E)$ and $\bV^+ \in \MRV(\alpha_{0},b_{0},\nu_{0}^{+},\E)$. Moreover, conditions (A2), (A3) and (A4) for Model \ref{additive:model} with $\bY^+$ and $\bV^+$ are also satisfied. Hence, $\bZ^+$ can be considered to be an example of  case (a) again.
\item[(c)] \textbf{Step 1.} First, we show that $\bZ^{\min}=(Z_1,\min(Z_1,Z_2))\in \MRV({\alpha}, b, \nu^{\min}, \E)$. Define
\begin{eqnarray*}
    \bR:=\bZ^{\min}-(Y_{1},0)=(V_1,\min(Y_1+V_1,Y_2+V_2)).
\end{eqnarray*}
Then
\begin{eqnarray*}
    \Pr(\|\bR\|>t)\leq \Pr(\min(Y_1,Y_2)>t/2)+\Pr(\max(V_1,V_2)>t/2)=o(\Pr(Y_1>t)) \quad \text{as } t\to\infty,
\end{eqnarray*}
since $\ov F_{\min(Y_{1},Y_2)} (t)\in \RV_{-(\alpha+\alpha^{*})}$ or  $\ov F_{\min(Y_{1},Y_2)}(t) = o(\ov F_{Y_1}(t)t^{-\alpha^{*}})$ and $\ov F_{\max(V_{1},V_2)}(t) \in \RV_{-\alpha_0}$.  Since \linebreak $(Y_1,0)\in \MRV({\alpha}, b, \nu^{\min}, \E)$,
using \cite[Lemma 3.12]{jessen:mikosch} we have  $\bZ^{\min}\in \MRV({\alpha}, b, \nu^{\min}, \E)$.

\textbf{Step 2.} Next, we prove $\bZ^{\min}\in \MRV({\alpha}_0, b_0, \nu^{\min}_0, \E_0)$.
We apply criterion (ii) of the Portmanteau Theorem 2.1 in \citet{lindskog:resnick:roy:2014} to show that
\[\nu_{t}(\cdot) = t\Pr\left({\bZ^{\min}}/{b_{0}(t)} \in \cdot \right) \stackrel{\M}{\to} \nu_{0}^{\min}(\cdot) \quad \text{in} \;\; \M(\E_{0}).\] Let $f$ be in $ C((0,\infty)^2)$
and without loss of generality suppose that $f$ is bounded by a constant $\|f\|$, is uniformly continuous and
\begin{eqnarray*}
    f(\bx)=0\quad \text{ if } x_1\wedge x_2<\eta,
\end{eqnarray*}
for some $\eta>0$. Uniform continuity of $f$ means that the modulus of continuity
\begin{eqnarray*}
    \omega_f(\delta):=\sup\{|f(\bx)-f(\by)|: \|\bx-\by\|<\delta\}\stackrel{\delta\to 0}{\to}0.
\end{eqnarray*}
Since $\bV^{\min}:=(V_1,\min(V_1,V_2)) \in \MRV(\alpha_0,b_0,\nu_0^{\min},\E)$ we have
\begin{eqnarray*}
    \lim_{t\to\infty}t\,\E\{f(\bV^{\min}/b_0(t))\}=\nu_0^{\min}(f),
\end{eqnarray*}
and so it suffices to show that as $t\to\infty$,
\begin{eqnarray*} 
  \lim_{t\to\infty}t\,\E\{f(\bZ^{\min}/b_0(t))\} - t\,\E\{f(\bV^{\min}/b_0(t))\} =0.\end{eqnarray*}
Let $0<\delta<\eta$. Then
\begin{align*}
  t\,\E\{f(\bZ^{\min}/b_0(t))\} & - t\,\E\{f(\bV^{\min}/b_0(t))\} \\
        =&\, t \,\E\left[\left\{f(\bZ^{\min}/b_0(t))-f(\bV^{\min}/b_0(t))\right\}\1_{\{Y_1\vee Y_2>b_0(t)\delta\}}\right]\\
        \quad& +\, t\,\E\left[\left\{f(\bZ^{\min}/b_0(t))-f(\bV^{\min}/b_0(t))\right\}\1_{\{Y_1\vee Y_2\leq b_0(t)\delta\}}\right] \\
        =:&\,  I_1(t,\delta)+I_2(t,\delta).
\end{align*}
For $I_1(t,\delta)$ we take the upper bound
\begin{eqnarray*}
    \hspace*{-1cm}|I_1(t,\delta)|
    &\leq&t\,\left|\E\left[\left\{f(\bZ^{\min}/b_0(t))-f(\bV^{\min}/b_0(t))\right\}\1_{\{Y_1\wedge Y_2>b_0(t)\delta\}}\right]\right| \\
        &&+t\,\left|\E\left[\left\{f(\bZ^{\min}/b_0(t))-f(\bV^{\min}/b_0(t))\right\}\1_{\{Y_1>b_0(t)\delta, Y_2\leq b_0(t)\delta\}}\right]\right| \\
        &&+t\,\left|\E\left[\left\{f(\bZ^{\min}/b_0(t))-f(\bV^{\min}/b_0(t))\right\}\1_{\{Y_1\leq b_0(t)\delta,Y_2>b_0(t)\delta\}}\right]\right| \\
        &=:&J_1(t,\delta)+J_2(t,\delta)+J_3(t,\delta).
\end{eqnarray*}
First, by Potter's Theorem  (see \citet[Theorem 1.5.6]{bingham:goldie:teugels:1989}) for any $\epsilon\in(0,\alpha^*+\alpha-\alpha_0)$ there exists a constant $C_{1}>0$ such that for large $t$
\begin{align*}
    J_1(t,\delta) & \leq 2\|f\| t \Pr(Y_1\wedge Y_2>b_0(t)\delta)\\
                        & =2\|f\|t\,\Pr(Y_1>b_0(t)\delta) \Pr( Y_2>b_0(t)\delta)
                        \leq C_{1} t^{\frac{\alpha_0-\alpha-\alpha^*+\epsilon}{\alpha_0}}
    \stackrel{t\to\infty}{\to}0.
\end{align*}
Similarly, by Potter's Theorem  for $\epsilon\in(0,\alpha)$ there exists a constant $C_{2}>0$ such that
\begin{align*}
    J_2(t,\delta) & \leq 2\|f\| t \Pr(Y_1\wedge V_2>b_0(t)\min\{(\eta-\delta),\delta\}) 
                        \leq C_{2} t^{\frac{-\alpha+\epsilon}{\alpha_0}}
    \stackrel{t\to\infty}{\to}0.
\end{align*}
Finally, again,  by Potter's Theorem  for $\epsilon\in(0,\alpha^{*})$ there exists a constant $C_{3}>0$ such that
\begin{align*}
    J_3(t,\delta) & \leq 2\|f\| t \Pr(V_1\wedge Y_2>b_0(t)\min\{(\eta-\delta),\delta\}) 
                        \leq C_{3} t^{\frac{-\alpha^{*}+\epsilon}{\alpha_0}}
    \stackrel{t\to\infty}{\to}0.
\end{align*}
Hence, we have
\begin{eqnarray*}
    \lim_{t\to\infty}|I_1(t,\delta)|=0.
\end{eqnarray*}
Since by definition $f(\bx)=0$ if $x_{1}\wedge x_{2}<\eta$, we have
\begin{eqnarray*}
    |I_2(t,\delta)|&=&t\left|\E\left[\left\{f(\bZ^{\min}/b_0(t))-f(\bV^{\min}/b_0(t))\right\}\1_{\{Y_1\vee Y_2\leq b_0(t)\delta,V_1\wedge V_2>(\eta-\delta)b_0(t)\}}\right]\right| \\
        &\leq&\omega_f(\delta)t\Pr(V_1\wedge V_2>(\eta-\delta)b_0(t)).
\end{eqnarray*}
Hence,
\begin{eqnarray*}
    \lim_{\delta\to 0}\lim_{t\to\infty}|I_2(t,\delta)|\leq \limsup_{\delta\to 0}\omega_f(\delta)(\eta-\delta)^{-\alpha_0}=0.
\end{eqnarray*}
Therefore, we have  $\bZ^{\min}=(Z_1,\min(Z_1,Z_2))\in\MRV({\alpha_0}, b_0, \nu_0^{\min}, \E_0)$.

\item[(d)] \textbf{Step 1.} Note that $Y_2 \equiv  0$. First, we show that $\bZ^{\max}=(\max(Z_1,Z_2),Z_2)\in \MRV({\alpha}, b, \nu,
    \E)$. Define \linebreak $\bR:=\bZ^{\max}-\bY.$
Then
\begin{eqnarray*}
    \Pr(\|\bR\|>t)\leq \Pr(\max(V_1,V_2)>t/2)=o(\Pr(Y_1>t)) \quad \text{as } t\to\infty.
\end{eqnarray*}
Since $\bY=(Y_1,0)\in \MRV({\alpha}, b, \nu^{\max}, \E)$,
using \cite[Lemma 3.12]{jessen:mikosch} we have  $\bZ^{\max}\in \MRV({\alpha}, b, \nu^{\max}, \E)$.

\textbf{Step 2.} Next, we show that $\bZ^{\text{max}}=(\max(Z_1,Z_2),Z_2)\in \MRV({\alpha_0}, b_0, \nu_0^{\max}, \E_0)$.\\
Since $\bV \in \MRV(\alpha_0,b_0,\nu_0,\E)$ we have $\bV^{\max} := (\max(V_{1}, V_{2}), V_{2}) \in \MRV(\alpha_0,b_0,\nu_0^{\max},\E)$ and with the notations and definitions of (c),
\begin{eqnarray*}
    \lim_{t\to\infty}t\,\E\{f(\bV^{\max}/b_0(t))\}=\nu_0^{\max}(f).
\end{eqnarray*}
So it suffices to show that
\begin{eqnarray*} 
  \lim_{t\to\infty}t\,\E\{f(\bZ^{\max}/b_0(t))\} - t\,\E\{f(\bV^{\max}/b_0(t))\} =0.\end{eqnarray*}
Let $0<\delta<\eta$. Then
\begin{align*}
 t\,\E\{f(\bZ^{\max}/b_0(t))-f(\bV^{\max}/b_0(t))\}
        =&\, t \,\E\left[\left\{f(\bZ^{\max}/b_0(t))-f(\bV^{\max}/b_0(t))\right\}\1_{\{Y_1>b_0(t)\delta\}}\right]\\
        \quad& +\, t\,\E\left[\left\{f(\bZ^{\max}/b_0(t))-f(\bV^{\max}/b_0(t))\right\}\1_{\{Y_1\leq b_0(t)\delta\}}\right] \\
        =:&\,  I_1(t,\delta)+I_2(t,\delta).
\end{align*}
Again by Potter's Theorem  for any fixed  $\epsilon\in(0,\alpha)$ there exists a constant $C>0$ such that for large $t$,
\begin{eqnarray*}
    |I_1(t,\delta)|\leq 2\|f\| t \Pr(Y_1>b_0(t)\delta)\Pr(V_2>b_0(t)\eta) \le C t^{\frac{-\alpha+\epsilon}{\alpha_{0}}}
    \stackrel{t\to\infty}{\to}0,
\end{eqnarray*}
and moreover,
\begin{eqnarray*}
    |I_2(t,\delta)|&=&t\left|\E\left[\left\{f(\bZ^{\max}/b_0(t))-f(\bV^{\max}/b_0(t))\right\}\1_{\{Y_1\leq b_0(t)\delta,V_2>\eta b_0(t)\}}\right]\right| \\
        &\leq&\omega_f(\delta)t\Pr(V_2> \eta b_0(t))
\end{eqnarray*}
implying $$\lim_{\delta\to 0}\lim_{t\to\infty}|I_2(t,\delta)|\leq \limsup_{\delta\to 0}\omega_f(\delta)\eta^{-\alpha_0}=0.$$
Hence, we can conclude as in (c) that  $\bZ^{\max}=(\max(Z_1,Z_2),Z_2)\in \MRV({\alpha_0}, b_0, \nu_0^{\max}, \E_0)$.

\end{enumerate}

\end{proof}

Since in Model \ref{additive:model} we know either $\ov F_{Y_2}\in\RV_{-\alpha}$, or $\ov F_{Y_1}(t) \in \RV_{-\alpha}$ where $\lim_{t\to\infty} {\ov F_{Y_2}(t)}/{\ov F_{Y_1}(t)}=0$,
 we obtain the following special case as well.

\begin{Corollary} \label{Corollary 2a}
Let  $\bZ=\bY+\bV$ be as in Model \ref{additive:model} and assume
$2\alpha>\alpha_0$.
Then the assumptions  and hence, the conclusions of \cref{theorem:3.1} are satisfied.
\end{Corollary}

\subsection{Hidden regular variation and the survival copula} \label{subsec:copsurv}


We present a characterization of hidden regular variation
via the behavior of the survival copula. First, we introduce a generalized version of the upper tail order function
(see \citet{Hua:Joe:2011,Hua:Joe:2013b}) along with an \emph{upper tail order pair}. The notion of upper tail order pair is related also to
 \emph{operator tail dependence} in \citet{li:2016}, and to the
\emph{generalized upper tail index} $\kappa$ in \citet{wadsworth:tawn:2013}.

\begin{Definition}[Upper Tail Order]
Let $F$ be a bivariate distribution function with survival copula $\wh C$. For a given constant $\tau>0$, if there exist a real constant
$\kappa>0$, and a slowly varying function
$\ell$ at $0$ with
\begin{eqnarray}\label{eq:utpair}
    \wh C(s,s^\tau)\sim s^{\kappa}\ell(s) \quad \text{as } s\downarrow 0,
\end{eqnarray}
the pair $(\kappa,\tau)$ is called an {\em upper tail order pair} of $F$.
The {\rm upper tail order function} $T:\E_0\to\R_+$ with respect to $(\kappa,\tau)$ is defined as
\begin{eqnarray}\label{eq:Txy}
    T(x,y)=\lim_{s\downarrow 0}\frac{\wh C(sx,s^\tau y)}{s^{\kappa}\ell(s)} \quad \text{for } x,y>0
\end{eqnarray}
provided that the limit function exists.
\end{Definition}
\begin{Remark}
Note that the pair $(\tau, \kappa)$ need not be a unique  for the definition to hold. Albeit this fact, introducing the quantity $\tau$ helps in rescaling
marginal tails when they are not equivalent (see Theorem \ref{Lemma:Hua:Joe:Li:1} below). Since \linebreak $0\leq \wh C(s,s^\tau) \leq \wh C(1,s^\tau)= s^\tau$ for $s\in(0,1)$ we have $\kappa\geq \tau$
and similarly we obtain $\kappa\geq 1$ as well.
Note that the existence of the upper tail order pair is not a
sufficient assumption for the existence of the  upper tail order function. We often provide examples fixing $\tau=1$, which also fixes the value of $\kappa$.
\end{Remark}
\begin{Lemma}   \label{Lemma 3.6}
Suppose the bivariate distribution function $F$ with survival copula $\wh C$ exhibits asymptotic upper tail independence and the upper tail order pair $(\kappa,\tau)$ exists with
$\tau\geq 1$ and $\wh C(s,s^\tau)\sim s^{\kappa}\ell(s)$ as $s\downarrow 0$. Then $$\lim_{s\downarrow 0}s^{\kappa-1}\ell (s)=0.$$
\end{Lemma}
\begin{proof}
Note that,  for $\tau\geq 1$,
\begin{eqnarray}\label{ineq}
    1= \lim_{s\downarrow 0}\frac{\wh C(s,s^\tau)}{s^{\kappa}\ell(s)}\leq \lim_{s\to0 } \frac{\wh C(s,s)}{s}
    \liminf_{s\downarrow 0}\frac{1}{{s^{\kappa-1}\ell(s)}}.
\end{eqnarray}
Since $F$ exhibits asymptotic upper tail independence, using Lemma \ref{lem:basicconnect}, we have  $\lim_{s\downarrow 0}\wh C(s,s)/s=0.$
Hence, the inequality in \eqref{ineq} is only possible if $\liminf_{s\downarrow 0}1/(s^{\kappa-1}\ell(s))=\infty$.
\end{proof}

\begin{Remark}
The classical definition of upper tail order is for $\tau=1$  and it is
equivalent to the definition of \textit{coefficient of tail
dependence} in \citet{Ledford:Tawn}.
\begin{enumerate}[(1)]
\item If $\kappa=\tau=1$ and $\lim_{s\downarrow 0}\ell(s)=c$ for some finite constant $c$, then we get
asymptotic dependence in the upper tail. In this case $T$ is the upper tail dependence function introduced
in \citet{Jaworski2006}.  However, if $\kappa=\tau=1$ and $\lim_{s\downarrow 0}\ell(s)=0$ then again we observe asymptotic tail independence.
\item The case $1<\kappa<2$, $\tau=1$ is between tail dependence ($\kappa=1$ and $c\not=0$) and tail independence ($\kappa=2$) and indicates some positive tail dependence
although the tails are asymptotically tail independent. It is called   \textit{intermediate tail dependence} by \citet{Hua:Joe:2011,Hua:Joe:2013b}.
\item Note that it is possible to have $\kappa>2$ which often signifies \emph{negative tail dependence}; see Example \ref{example:cop} (a) below.
 \end{enumerate}
\end{Remark}


\begin{Example} \label{example:cop} We compute upper tail order functions and upper tail order pairs for a few well-known (survival) copula models here.
\begin{itemize}
\item[(a)] The \emph{Gaussian copula} turns out to be one of the most famous, if not infamous copula models, especially in \emph{financial risk management};
see \citet{salmon:2009}. It is given by
\begin{eqnarray*}
    C_{\Phi,\rho}(u,v)=\Phi_2(\Phi^{\leftarrow}(u),\Phi^{\leftarrow}(v)) \quad \text{for } (u,v)\in[0,1]^2,
\end{eqnarray*}
where $\Phi$ is the standard-normal distribution function and $\Phi_2$ is a bivariate normal distribution function
with standard normally distributed margins and correlation $\rho$. Then  the survival copula  satisfies:
\begin{eqnarray*}
    \wh C_{\Phi,\rho}(s,s)=C_{\Phi,\rho}(s,s)\sim s^{\frac{2}{\rho+1}}\ell(s) \quad \text{as } s\downarrow 0.
\end{eqnarray*}
(see \citet{reiss:1989,Ledford:Tawn}). For $\rho\in(-1,1)$ we have $\kappa=2/(\rho+1)>1$ and $\tau=1$ so that a distribution  with Gaussian copula
has still some kind of dependence although it exhibits asymptotic upper tail independence. The upper tail order function
is given by (see \citet[Example 7.2.7]{reiss:1989})
\begin{eqnarray*}
    T(x,y)= x^{\frac{1}{\rho+1}}y^{\frac{1}{\rho+1}} \quad \text{for } x,y>0.
\end{eqnarray*}
Note that $0<\rho <1$ implies $1<\kappa<2$ relating to positive intermediate tail dependence, $\rho=0$ implies $\kappa=2$ which is the independent case
and $\rho<0$ implies $\kappa>2$ which is the case of \emph{negative tail dependence}.

\item[(b)] If the survival copula  is a \emph{Marshall-Olkin copula} then:
$$\wh C_{\gamma_1,\gamma_2}(u,v)=uv\min(u^{-\gamma_1},v^{-\gamma_2})  \quad \text{for } (u,v)\in[0,1]^2,$$  some $\gamma_1,\gamma_2\in(0,1)$.
For a fixed $\tau\ge 1$. the upper tail order is {$\kappa=\max(\tau+1-\gamma_1,\tau+1-\tau\gamma_2)$}, and the upper tail order
function is
$$T(x,y)=\left\{\begin{array}{ll}
    xy^{1-\tau \gamma_2}, & \text{if } \gamma_1>\tau\gamma_2,\\
    xy\max(x,y^{1/\tau})^{-\gamma_1},& \text{if }\gamma_1=\tau\gamma_2,\\
    x^{1-\gamma_1}y, & \text{if }\gamma_1<\tau\gamma_2,
\end{array}    \right. \quad \text{for } x,y>0.
    $$
The Marshall-Olkin copula
belongs to the class of extreme value copulas. This structure of $T(x,y)$ holds in general for bivariate
extreme value copulas with discrete Pickands dependence function; see \citet[Example 2]{Hua:Joe:2011}.

\item[(c)] If the survival copula is a \emph{Morgenstern copula} with parameter $-1\le \theta\leq 1$ then:
\begin{eqnarray*}
    \wh C_\theta(u,v)=uv(1+\theta(1-u)(1-v)) \quad \text{for } (u,v)\in[0,1]^2.
\end{eqnarray*}
Hence, for $-1<\theta\le 1$ and fixed $\tau\ge 1$ we get $\kappa=\tau+1$ with upper tail order function $T(x,y)=xy^{\tau}$ for $x,y>0$. For $\theta=-1$ we have
the upper tail order pair $(\kappa=3,\tau=1)$ with upper tail order function $T(x,y)=xy(x+y)$ for $x,y>0$. If we fix $\tau>1$ then $\kappa=1+2\tau$ and
the upper tail order function is $T(x,y)=xy^{2}$ for $x,y>0$.
\item[(d)] A copula is called an \emph{Archimedean copula}  if it is  of the form
\begin{eqnarray*}
    C(u,v)=\phi^{\leftarrow}(\phi(u)+\phi(v))\quad \mbox{ for }(u,v)\in[0,1]^2,
\end{eqnarray*}
where the Archimedean generator $\phi:[0,1]\to[0,1]$ is convex, decreasing and satisfies $\phi(1)=0$.
In the bivariate case this is a necessary and sufficient condition for $C$ to be a copula. If the generator
$\phi$ is regularly varying near $0$ the lower tail is asymptotically independent and if $\phi$ is regularly varying at $1$
the upper tail is asymptotically  independent (see \citet{Ballerini,Capera:Fougeres:Genest}). \citet[Section 4]{Charpentier:Segers} provide
tail order coefficients and functions for Archimedean copulas specifically under asymptotic tail independence. We use them to generalize results
 for the tail order pair here; see also \citet{Hua:Joe:2011}.
\begin{enumerate}
       \item Let $\phi^{\leftarrow}$ be twice continuously differentiable with ${\phi^{\leftarrow}}''(0)<\infty$.
            Then an upper tail order pair is $(\kappa=2,\tau=1)$ with upper tail order function
            \begin{eqnarray*}
                T(x,y)=xy.
            \end{eqnarray*}
       \item Let $\phi'(1)=0$ and let the function $-\phi'(1-s)-s^{-1}\phi(1-s)$ be positive and slowly varying around $0$.
             Then the upper tail order pair is $(\kappa=1,\tau=1)$ with upper tail order function
            \begin{eqnarray*}
                T(x,y)=\{(x+y)\log(x+y)-x\log x-y\log y\}/(2\log 2).
            \end{eqnarray*}
       \item Assume that $\widehat C$ is an Archimedean copula with generator $\phi$ satisfying $\phi(0)=\infty$, $\lim_{s\downarrow 0}s\phi'(s)/\phi(s)=0$
            and $-1/(\log \phi^{\leftarrow})'$ is regularly varying with some index $\rho\leq 1$.  Then the upper tail order pair is
            $(\kappa=2^{1-\rho},\tau=1)$ with upper tail order function
            \begin{eqnarray*}
                T(x,y)=x^{2^{-\rho}}y^{2^{-\rho}}.
            \end{eqnarray*}
\end{enumerate}
\end{itemize}
\end{Example}

We are now able to present a connection between hidden regular variation and the upper tail order
function;  these results are extensions of \citet{Hua:Joe:Li:2014} and \citet{Li:Hua:2015} which include several examples
as well. The first result is a generalization of \citet[Proposition 3.2]{Hua:Joe:Li:2014} for $\tau=1$.

\begin{Theorem} \label{Lemma:Hua:Joe:Li:1}
Let $\bZ\in \MRV(\alpha, b,\nu, \E)\cap \HRV(\alpha_0, b_0,\nu_0)$ with continuous margins $F_1, F_2$ satisfying
\linebreak $\lim_{t\to\infty}\ov F_{2}(t)/\ov F_{1}^\tau(t)=\eta\in(0,\infty)$ where $\alpha_0/\alpha\geq \tau \geq 1$.
 Then an upper tail order pair
exists with $(\kappa=\alpha_0/\alpha,\tau)$, and the corresponding upper tail order function is given by
\begin{eqnarray*}
    T(x,y)=C\nu_0\left(\left(x^{-1/\alpha},\infty\right]\times\left(\eta^{1/(\tau\alpha)}y^{-1/(\tau\alpha)},\infty\right]\right)  \quad  \text{for } x,y>0,
\end{eqnarray*}
where $0<C= \left\{\nu_0\left(\left(1,\infty\right]\times\left(\eta^{1/(\tau\alpha)},\infty\right]\right)\right\}^{-1}<\infty$.
\end{Theorem}
\begin{proof}
The proof follows with similar arguments as in  \citet[Proposition 3.2]{Hua:Joe:Li:2014} if we can show the following:
\begin{enumerate}
    \item[(a)] $\ov F_1\in\RV_{-\alpha}$.
    \item[(b)] If for any $y>0$ and $\epsilon>0$ small there exists an $s_0>0$ such that for any $0<s<s_0$ the inequality
    \begin{eqnarray*}
        (1-\epsilon)\eta^{1/(\tau\alpha)}y^{-1/(\tau\alpha)}<\frac{\ov F_2^{\leftarrow}(ys^{\tau})}{\ov F_1^{\leftarrow}(s)}<(1+\epsilon)\eta^{1/(\tau\alpha)}y^{-1/(\tau\alpha)}
    \end{eqnarray*}
    holds.
    \end{enumerate}
We prove (a) and (b) in the following.
\begin{enumerate}
    \item[(a)] \, Case 1: $\tau>1$. Then $\lim_{t\to\infty}\ov F_2(t)/\ov F_1(t)=0$. Since $\bZ=(Z_1,Z_2)\in \MRV(\alpha, b,\nu, \E)$
    in particular, \linebreak $\max(Z_1,Z_2)\in\RV_{-\alpha}$. Moreover,
    \begin{eqnarray*}
        1\leq \frac{\Pr(\max(Z_1,Z_2)>t)}{\ov F_1(t)}\leq 1+\frac{\ov F_2(t)}{\ov F_1(t)}\stackrel{t\to\infty}{\to}1.
    \end{eqnarray*}
    Hence, $\ov F_1\in\RV_{-\alpha}$.\\
    Case 2: $\tau=1$.  Since $\bZ=(Z_1,Z_2)\in \MRV(\alpha, b,\nu, \E)$  either $\ov F_1\in\RV_{-\alpha}$ or $\ov F_2\in\RV_{-\alpha}$. Since \linebreak
    $\lim_{t\to\infty}\ov F_{2}(t)/\ov F_{1}(t)=\eta\in(0,\infty)$ we have that $\ov F_1$ and $\ov F_2$ are tail-equivalent and in particular, both are $\RV_{-\alpha}$.
    \item[(b)] \, Let $y>0$. Define $G_1(x)=1/\overline{F}_2(x)$ and $G_2(x)=1/\{\eta\overline F_1^\tau(x)\}$. Then $G_1(x)\sim G_2(x)$ as $x\to\infty$, $G_1,G_2\in\RV_{\tau\alpha}$
    and both are non-decreasing.
    Using \citet[Proposition 2.6]{Resnick:2007} we have $G_1^{\leftarrow},G_2^{\leftarrow}\in\RV_{1/(\alpha\tau)}$ and $G_1^{\leftarrow}(z)\sim G_2^{\leftarrow}(z)$
    as $z\to\infty$. Here $G_1^{\leftarrow}(z)=\overline{F}_2^{\leftarrow}(1/z)$ and $G_2^{\leftarrow}(z)=\overline{F}_1^{\leftarrow}(1/(\eta z)^{1/\tau})$. Hence,
    \begin{eqnarray} \label{s1}
        \overline{F}_2^{\leftarrow}(s)\sim \overline{F}_1^{\leftarrow}(\eta^{-1/\tau}s^{1/\tau}) \quad \text{ as } s\downarrow 0.
    \end{eqnarray}
Therefore,
    \begin{eqnarray} \label{s2}
        \frac{\ov F_2^{\leftarrow}(ys^{\tau})}{\ov F_1^{\leftarrow}(s)} =\frac{\ov F_2^{\leftarrow}(ys^{\tau})}{\ov F_1^{\leftarrow}(\eta^{-1/\tau}y^{1/\tau}s)}
            \frac{\ov F_1^{\leftarrow}(\eta^{-1/\tau}y^{1/\tau}s)}{\ov F_1^{\leftarrow}(s)}\stackrel{s\downarrow 0}{\to} \eta^{1/(\tau\alpha)}y^{-1/(\tau\alpha)},
    \end{eqnarray}
   where  the first term converges to $1$ due to \eqref{s1} and the second term converges to  $\eta^{1/(\tau\alpha)}y^{-1/(\tau\alpha)}$ since  $\ov F_1^{\leftarrow}$ is regularly varying with index $-1/\alpha$ near $0$. Hence, (b) holds.
   \end{enumerate}
\end{proof}
\begin{Remark}
Note that in the presence of \emph{hidden regular variation} as above the pair $(b_0,\nu_0)$ is not exactly uniquely defined since if $\bZ\in \HRV(\alpha_0, b_0,\nu_0)$
then $\bZ\in \HRV(\alpha_0, Cb_0,C^{-\alpha_0}\nu_0)$ for $0<C<\infty$ as well. But we can consider this to be uniqueness up to a scale.
\end{Remark}

The converse of \Cref{Lemma:Hua:Joe:Li:1} also holds; this is an extension of {\citet[Proposition 3.3]{Hua:Joe:Li:2014}} for $\tau=1$.

\begin{Theorem}
 \label{Lemma 2.9}
Let $\bZ\in \left[0,\infty\right)^2$ with continuous margins $F_1, F_2$, survival copula $\wh C$, $\E|Z_1|<\infty$ and
 $\ov F_{1}\in\RV_{-\alpha}$. Suppose  $\wh C$ has upper tail order pair $(\kappa,\tau)$ with $\kappa\geq \tau\geq 1$ and some slowly varying function $\ell$ at $0$ with \linebreak
  $\lim_{s\downarrow 0}s^{\kappa-1}\ell(s)=0$ satisfying \eqref{eq:utpair} and \eqref{eq:Txy}.
Moreover, assume that
$\lim_{t\to\infty}\ov F_{2}(t)/\ov F_{1}^\tau(t)=\eta\in(0,\infty)$.  Then \linebreak $\bZ\in \MRV(\alpha, b,\nu, \E)\cap \HRV(\alpha_0, b_0,\nu_0)$ with $\alpha_0=\alpha\kappa$,
\begin{eqnarray*}
    \nu_0\left(\left(x,\infty\right]\times\left(y,\infty\right]\right)=T(x^{-\alpha},\eta^{\tau\alpha}y^{-\tau\alpha}) \quad \text{for } x,y>0,
\end{eqnarray*}
and some properly chosen $b_0\in\RV_{1/\alpha_0}$.
\end{Theorem}

\begin{proof}
The proof follows similar arguments as that of \citet[Proposition 3.2]{Hua:Joe:Li:2014} using \eqref{s2} and is omitted here.
\end{proof}

\begin{Remark}
It is possible to have $\kappa=\tau=1$ as well  but then $\lim_{s\downarrow 0}\ell(s)=0$ which excludes the case of asymptotic upper tail dependence.

\end{Remark}
A conclusion from these results is that  hidden regular variation is not only the effect of  the copula of the joint distribution but also of the ratio of the
individual marginal tails.
This may seem surprising since copulas, in theory, are supposed to decouple the marginal distributions from  the dependence
structure of random vectors. Clearly, this is not the case for tail dependence especially for regularly varying tails as observed here.

\section{Asymptotic behavior of risk measures} \label{sec:MMEMES}

The asymptotic behavior of the MME and the MES under  hidden regular variation were obtained in  \citet{DasFasen2017}.
 Under these constraints consistent estimators for MME and MES were derived; moreover \citet{CaiMusta2017}  have also shown asymptotic normality for MES.
 In this section we present a variety of examples of model classes satisfying the assumptions of  \citet{DasFasen2017} and relate these
assumptions in particular, to copula models.
We recall in brief the results from \cite{DasFasen2017} and then present additive models in \Cref{Section 4.1}
and copula models in \Cref{Section 4.2} satisfying these assumptions.  Eventually, we note that the assumption of HRV is not always necessary;
using the theory of copulas we  generalize the results of \cite{DasFasen2017}
beyond any assumption of HRV.

\begin{Theorem}[\citet{DasFasen2017}, Theorems 1 and 2]\label{thm:MES}\label{theorem:4.1}
Let $\bZ =(Z_1, Z_2)\in \MRV({\alpha}, b, \nu, \E)\cap $ \linebreak $\HRV({\alpha_0}, b_0, \nu_0, \E_0)$ be in  $ \left[0,\infty\right)^2$  with $\alpha_0\ge \alpha\geq 1$ and $\E|Z_1|<\infty$.
\begin{itemize}
    \item[(a)]  Suppose the following condition holds:
        \begin{equation} \label{C1}
          \lim_{M\to\infty}\lim_{t\to\infty}\int_{M}^\infty \frac{\Pr(Z_1>xt, Z_2>t)}{\Pr(Z_1>t, Z_2>t)}\, \mathrm dx=0.  \tag{\text{\rm  B}}
        \end{equation}
       Then
        \begin{align*}
            \lim\limits_{p\downarrow 0} \frac{pb_0^{\la}\{\VaR_{1-p}(Z_2)\}}{\VaR_{1-p}(Z_2)} \MME(p) = \int_1^{\infty} \nu_0((x,\infty)\times(1,\infty))\;\mathrm dx.
        \end{align*}
        Moreover, $0<\int_1^{\infty}\nu_0((x,\infty)\times(1,\infty))\;\mathrm dx<\infty$.
    \item[(b)] Suppose the following condition holds:
     \begin{equation} \label{cond:dct1}
          \lim_{M\to\infty}\lim_{t\to\infty}\left[\int_{M}^\infty +\int_0^{1/M}\right] \frac{\Pr(Z_1>xt, Z_2>t)}{\Pr(Z_1>t, Z_2>t)}\, \mathrm dx=0.  \tag{\text{\rm C}}
     \end{equation}
     Then
     \begin{align*}
        \lim\limits_{p\downarrow 0} \frac{pb_0^{\la}\{\VaR_{1-p}(Z_2)\}}{\VaR_{1-p}(Z_2)} \MES(p) = \int_0^{\infty} \nu_0((x,\infty)\times(1,\infty))\;\mathrm dx.
       \end{align*}
Moreover, $0<\int_0^{\infty}\nu_0((x,\infty)\times(1,\infty))\;\mathrm dx<\infty$.
\end{itemize}
\end{Theorem}

Clearly, condition \eqref{cond:dct1} implies condition \eqref{C1}. In the face of it,  it appears that the rate of increase of  MES which is governed  by the function
$ pb_0^{\la}\{\VaR_{1-p}(Z_2)\}/\VaR_{1-p}(Z_2)$ is determined by the tail behavior of the marginal distribution $F_2$. However, we notice in \Cref{Section 4.2} that this is not true for MES; the rate is in fact governed by the joint tail behavior of the  copula of $(Z_1,Z_2)$ and that of the marginal tail of $F_1$.

\begin{Remark}  \label{Remark 4.2}$\mbox{}$ Define the function governing the limit behavior of $\MES(p)$ and $\MME(p)$ in Theorem \ref{theorem:4.1} as
\begin{eqnarray} \label{function a}
     a(t): =  \frac{b_0^{\la}\{\VaR_{1-1/t}(Z_2)\}}{t\VaR_{1-1/t}(Z_2)}.
\end{eqnarray}
\begin{enumerate}[(1)]
     \item From \citet[Remark 6]{DasFasen2017} we know that under condition \eqref{cond:dct1} the limit $\lim_{p\downarrow 0} a(1/p)=0$
        is valid and hence, under the conditions of Theorem \ref{theorem:4.1}, $\lim_{p\downarrow 0}\MES(p)=\infty$. Thus, even under the presence of asymptotic
        upper tail independence, the tail dependence
        is still strong enough for $\lim_{p\downarrow 0}\MES(p)=\infty$. In contrast, if $Z_1$ and $Z_2$ are independent
        we have $\MES(p)=\E(Z_1)$ (then condition \eqref{cond:dct1} is not satisfied).
   \item Consider the case $\overline{F}_{Z_2}\in\RV_{-\alpha^*}$. Then $(Z_1, Z_2)\in \MRV({\alpha}, b, \nu, \E)$ implies $\alpha^*\geq \alpha$.
    Furthermore, if additionally $(Z_1, Z_2) \in\MRV({\alpha_0}, b_0, \nu_0, \E_0)$ then $\alpha^*\leq \alpha_0$ as well (see  \cite[Lemma 2.7]{DasFasen2017}).
    In this case,
        $a(t) \in\RV_{(\alpha_0-\alpha^*-1)/\alpha^*}$.  Therefore,
         a necessary condition for   $\lim_{p\downarrow 0} a(1/p)=0$ is $\alpha_0\leq \alpha^*+1$ and a sufficient condition is
        $\alpha_0<\alpha^*+1$. Finally, $\alpha_0\leq\alpha^*+1$ is as well a necessary assumption for \eqref{cond:dct1}.
     \item If $Z_1$ and $Z_2$ are independent then $\alpha_0=\alpha+\alpha^*$ and hence, $\alpha\geq 1$ and $\alpha_0<\alpha^*+1$ is not possible. Thus, the independent case
        does not satisfy \eqref{cond:dct1} and a scaled limit for $\MES(p)$ cannot be calculated using Theorem \ref{theorem:4.1}.
    \item Under condition \eqref{C1} both $\lim_{p\downarrow 0} a(1/p)=\infty$ and $\lim_{p\downarrow 0}\MME(p)=0$ are
        possible. For the independent  margin case the asymptotic behavior of $\MME(p)$ can still be calculated using Theorem \ref{theorem:4.1}, since with $\alpha>1$ and $Z_1,Z_2$
        independent, condition \eqref{C1} is satisfied.
\end{enumerate}
\end{Remark}

\subsection{Asymptotic behavior of MME and MES for additive models} \label{Section 4.1}
In Section \ref{subsec:addmod},  we introduced with Model \ref{additive:model} a general additive model for multivariate regular variation and discussed in Theorem \ref{theorem:3.1} the
existence of hidden regular variation.  Do such models satisfy conditions \eqref{cond:dct1} and hence \eqref{C1} as well? The
following result provides an answer.


\begin{Theorem}  \label{theorem:5.3}
Let $\bZ=\bY+\bV$ be as in Model \ref{additive:model}. Suppose that $\E|Z_1|<\infty$  and
 $\alpha\leq  \alpha_0< 1+\alpha^*$.
Then the following models satisfy condition \eqref{cond:dct1}:
\begin{itemize}
    \item[(a)] $\bZ=(Z_1,Z_2) $.
    \item[(b)] $\bZ^+=(Z_1+Z_2,Z_1)$ and $\wt\bZ^+=(Z_1,Z_1+Z_2)$ if $Y_2\equiv0$.
    \item[(c)] $\bZ^{\min}=(Z_1,\min(Z_1,Z_2))$  if $\liminf_{t\to\infty}\Pr(Y_1>t)/\Pr(Y_2>t)>0$.
    \item[(d)] $\bZ^{\max}=(\max(Z_1,Z_2),Z_2)$ and $\wt\bZ^{\max}=(Z_2,\max(Z_1,Z_2))$ if  $Y_2\equiv0$.
    \end{itemize}
    As a consequence, the asymptotic limits of the  scaled $\MME(p)$ and $\MES(p)$ can be obtained in each case using Theorem \ref{theorem:4.1}.
\end{Theorem}

A direct consequence of \Cref{theorem:5.3} is the ability to compute asymptotic limits of $\MES^{+}, \MES^{\min}$ and $\MES^{\max}$ as defined in \eqref{I1} and is
summarized in the following corollary.

\begin{Corollary}
Let $\bZ=\bY+\bV$ be as in Model \ref{additive:model} and $a(t)$ be defined as in \eqref{function a}. Suppose that $\E|Z_1|<\infty$  and
 $\alpha\leq  \alpha_0< 1+\alpha^*$. Then there exist finite constants $K^+,K^{\min},K^{\max}>0$ so that following statements hold.
 \begin{enumerate}
    \item[(a)] $\lim_{p\downarrow 0}a(1/p)\MES^{+}(p)=K^+$ if $Y_2\equiv 0$.
    \item[(b)] $\lim_{p\downarrow 0}a(1/p)\MES^{\min}(p)=K^{\min}$   if $\liminf_{t\to\infty}\Pr(Y_1>t)/\Pr(Y_2>t)>0$
    \item[(c)] $\lim_{p\downarrow 0}a(1/p)\MES^{\max}(p)=K^{\max}$  if $Y_2\equiv 0$.
 \end{enumerate}
\end{Corollary}

We prove some auxiliary results first which are used to prove  Theorem \ref{theorem:5.3}. The following proposition provides sufficient conditions
under Model \ref{additive:model} for condition \eqref{cond:dct1} to hold.

\begin{Proposition}\label{prop:add1}
Let $\bZ=\bY+\bV$ be as in Model \ref{additive:model}. Suppose that $\E|Z_1|<\infty$ implying $\alpha\geq 1$ and that the following conditions are satisfied:
\begin{enumerate}[(i)]
\item \label{cond:Y} $\bZ \in \MRV(\alpha_0,b_0,\nu_0,\E_0)$.
\item \label{cond:Y_2V} ${\displaystyle
    \lim\limits_{t\to\infty} \frac{\Pr(Y_2>t)}{t\Pr(V_1>t,V_2>t)}=0
 }$.
 \item  \label{cond:dct2}
 ${\displaystyle
    \lim_{M\to\infty}\lim_{t\to\infty}\left[\int_0^{1/M}+\int_M^\infty\right]\frac{\Pr(Y_1>tx, Y_2>t)}{\Pr(V_1>t, V_2>t)}\,\mathrm dx =0
}.$
\end{enumerate}
Then $\bZ$ satisfies condition \eqref{cond:dct1}.
\end{Proposition}
\begin{proof}
By the assumptions of Model \ref{additive:model} and Theorem \ref{theorem:3.1},
$\bV \in \MRV(\alpha_0,b_0,\nu_0,\E)$, $\bV$ does not possess asymptotic tail independence, and $\bZ \in \MRV(\alpha_0,b_0,\nu_0,\E_0)$. Hence,
\begin{eqnarray} \label{A.1}
    \Pr(V_1>t)\sim C_1 \Pr(V_2>t)\sim C_2 \Pr(V_1>t,V_2>t)\sim  C_3 \Pr(Z_1>t,Z_2>t) \quad \text{ as } t\to\infty,
\end{eqnarray}
for some finite constants $C_1,C_2,C_3>0$.
Let $0<\delta<1$. For some $M>0$ we have
\begin{eqnarray*}
    \lefteqn{\left[\int_0^{1/M}+\int_M^\infty\right]\frac{\Pr(Z_1>xt, Z_2>t)}{\Pr(Z_1>t, Z_2>t)}\, \mathrm dx}\\
        &&=\left[\int_0^{1/M}+\int_M^\infty\right]\frac{\Pr(Z_1>xt, Z_2>t,Y_1>x\delta t,Y_2> \delta t)}{\Pr(Z_1>t, Z_2>t)}\,\mathrm dx\\
        &&\quad +\left[\int_0^{1/M}+\int_M^\infty\right]\frac{\Pr(Z_1>xt, Z_2>t,Y_1>x\delta t,Y_2\leq \delta t)}{\Pr(Z_1>t, Z_2>t)}\,\mathrm dx\\
        &&\quad + \left[\int_0^{1/M}+\int_M^\infty\right]\frac{\Pr(Z_1>xt, Z_2>t,Y_1\leq x\delta t,Y_2> \delta t)}{\Pr(Z_1>t, Z_2>t)}\,\mathrm dx\\
        &&\quad +\left[\int_0^{1/M}+\int_M^\infty\right]\frac{\Pr(Z_1>xt, Z_2>t,Y_1\leq x\delta t,Y_2\leq \delta t)}{\Pr(Z_1>t, Z_2>t)}\,\mathrm dx\\
        &&=:I_1(M,t)+I_2(M,t)+I_3(M,t)+I_4(M,t).
\end{eqnarray*}
Now, we investigate  all four terms separately. Using \eqref{A.1} for large enough $t$,
\begin{eqnarray*}
    I_{1}(M,t)&\leq& \left[\int_0^{1/M}+\int_M^\infty\right]\frac{\Pr(Y_1>x\delta t,Y_2> \delta t)}{\Pr(Z_1>t, Z_2>t)}\,\mathrm dx\\
            &\leq& \frac{2C_{3}}{C_{2}} \frac{\Pr(V_1>\delta t, V_2>\delta t)}{\Pr(V_1>t, V_2>t)}
            \left[\int_0^{1/M}+\int_M^\infty\right]\frac{\Pr(Y_1>x\delta t,Y_2> \delta t)}{\Pr(V_1>\delta t, V_2>\delta t)}\,\mathrm dx.
\end{eqnarray*}
A consequence of assumption \eqref{cond:dct2} and $\bV \in \MRV(\alpha_0,b_0,\nu_0,\E)$ is that $\lim_{M\to\infty}\lim_{t\to\infty}I_1(M,t)=0$.
For the second term $I_2(M,t)$ we have by the independence of $Y_1$ and $V_2$, and by Potter's bound for some $0<\epsilon<\alpha-1$,  the upper bound
\begin{eqnarray*}
    I_2(M,t)&\leq& \left[\int_0^{1/M}+\int_M^\infty\right]\frac{\Pr(V_2>(1-\delta) t,Y_1>x\delta t)}{\Pr(Z_1>t, Z_2>t)} \, \mathrm dx\\
        &=& \frac{\Pr(V_2>(1-\delta)t )}{\Pr(Z_1>t, Z_2>t)}\left[\int_0^{1/M}+\int_M^\infty\right] \Pr(Y_1>x\delta t)\, \mathrm dx \\
        &\leq& \frac{\Pr(V_2>(1-\delta)t )}{\Pr(Z_1>t, Z_2>t)}\left[\frac{1}{M}+C t^{-\alpha+\epsilon}M^{-\alpha+1+\epsilon}\right].
\end{eqnarray*}
Now using  \eqref{A.1} results in $\lim_{M\to\infty}\lim_{t\to\infty}I_2(M,t)=0$.
The third term $I_3(M,t)$ satisfies
\begin{eqnarray*}
    I_3(M,t)&\leq&\left[\int_0^{1/M}+\int_M^\infty\right]\frac{\Pr(V_1>x(1-\delta)t, Y_2> \delta t)}{\Pr(Z_1>t, Z_2>t)}\, \mathrm dx\\
        &=&\frac{\Pr( Y_2> \delta t)}{(1-\delta)t\Pr(Z_1>t, Z_2>t)}\left[\int_0^{1/M}+\int_M^\infty\right](1-\delta)t\Pr(V_1>x(1-\delta)t)\, \mathrm dx\\
        &\leq& C\frac{\Pr( Y_2> \delta t)}{t\Pr(V_1>t, V_2>t)}\E|V_1|
        \stackrel{t\to\infty}{\to}0
\end{eqnarray*}
by assumption \eqref{cond:Y_2V} and $\E|V_1|\leq \E|Z_1|<\infty$.
Finally, using  \eqref{A.1} and by Potter's theorem for some $0<\epsilon<\alpha_0-1$, the last term $I_4(M,t)$ has  the upper bound
\begin{eqnarray*}
    I_4(M,t)&\leq& \left[\int_0^{1/M}+\int_M^\infty\right]\frac{\Pr(V_1>x(1-\delta)t,V_2>(1-\delta)t)}{\Pr(Z_1>t, Z_2>t)}\, \mathrm dx\\
        &\leq& C\frac{\Pr(V_2>(1-\delta)t)}{\Pr(V_2>t)}\frac{1}{M} + C\int_M^\infty\frac{\Pr(V_1>x(1-\delta)t)}{\Pr(V_1>t)}\, \mathrm dx\\
        &\leq& C  M^{-1}+C M^{-\alpha_0+1+\epsilon},
\end{eqnarray*}
which implies $\lim_{M\to\infty}\lim_{t\to\infty}I_4(M,t)=0$ as well.
\end{proof}

The following lemma provides necessary and sufficient conditions for assumption \eqref{cond:Y_2V} of \cref{prop:add1} to hold.

\begin{Lemma}  \label{auxiliary:lemma}
Let $\bZ=\bY+\bV$ be as in Model \ref{additive:model}. Suppose that $\ov F_{Y_2}\in\RV_{-\alpha^*}$ for some $\alpha^*\geq \alpha \geq 1$. Then
 $\alpha_0\leq 1+\alpha^*$ is a necessary condition and $\alpha_0< 1+\alpha^*$ is a sufficient condition for assumption \eqref{cond:Y_2V} of \cref{prop:add1}  to hold.
\end{Lemma}
\begin{proof} Since $\Pr(Y_2>t)/\{t\Pr(V_1>t,V_2>t)\}\in\RV_{-\alpha^*-1+\alpha_0}$,
assumption ~\eqref{cond:Y_2V} of \cref{prop:add1}  can only hold if $-\alpha^*-1+\alpha_0\leq 0$. On the other hand, $-\alpha^*-1+\alpha_0< 0$ implies  assumption \eqref{cond:Y_2V}.
\end{proof}
Now we are able to prove \Cref{theorem:5.3}.

\begin{proof}[Proof of \Cref{theorem:5.3}] $\mbox{}$
\begin{enumerate}[(a)]
\item
First of all, due to \Cref{theorem:3.1},  Assumption \ref{prop:add1}\eqref{cond:Y} is valid.
 Assumption \ref{prop:add1}~\eqref{cond:Y_2V} is already satisfied
due to \cref{auxiliary:lemma}.
Finally,  Assumption \ref{prop:add1}~\eqref{cond:dct2} follows from
\begin{eqnarray*}
    \left[\int_0^{1/M}+\int_M^\infty\right]\frac{\Pr(Y_1>tx, Y_2>t)}{\Pr(V_1>t, V_2>t)}\, \mathrm dx
    &=&\frac{\Pr(Y_2>t)}{\Pr(V_1>t, V_2>t)}\left[\int_0^{1/M}+\int_M^\infty\right]\Pr(Y_1>tx)\, \mathrm dx\\
    &\leq &\frac{\Pr(Y_2>t)}{t\Pr(V_1>t, V_2>t)}\E(Y_1)\stackrel{t\to\infty}{\to}0.
\end{eqnarray*}
Hence, we obtain that $\bZ$ satisfies \eqref{cond:dct1}. 
\item This can be seen as special case of (a) where
 \begin{align*}
  \bZ^+ & =(Z_1+Z_2,Z_1)=(Y_1,0)+(V_1+V_2,V_1)=:\bY^++\bV^+,\\
\wt\bZ^+&=(Z_1,Z_1+Z_2)=(0,Y_1)+(V_1,V_1+V_2)=:\wt\bY^++\wt\bV^+.
\end{align*}
\item   For $x>1$,
\begin{eqnarray*}
    \frac{\Pr(Z_1>xt,\min(Z_1,Z_2)>t)}{\Pr(Z_1>t,\min(Z_1,Z_2)>t)}\leq \frac{\Pr(Z_1>xt, Z_2>t)}{\Pr(Z_1>t,Z_2>t)}
\end{eqnarray*}
so that $\bZ^{\min}$ satisfies \eqref{C1} due to (a). For $x\leq 1$,
\begin{eqnarray*}
    \frac{\Pr(Z_1>xt,\min(Z_1,Z_2)>t)}{\Pr(Z_1>t,\min(Z_1,Z_2)>t)}=1
\end{eqnarray*}
so that \eqref{cond:dct1} holds for $\bZ^{\min}$ as well.
\item  For $x>0$,
\begin{align*}
    \frac{\Pr(\max(Z_1,Z_2)>xt,Z_2>t)}{\Pr(\max(Z_1,Z_2)>t,Z_2>t)}&\leq \frac{\Pr(Z_1+Z_2>xt,Z_2>t)}{\Pr(Z_2>t)}\\
        &=\frac{\Pr(Z_1+Z_2>xt,Z_2>t)}{\Pr(Z_1+Z_2>t,Z_2>t)}.
\end{align*}
Then $\bZ^{\max}$ satisfies \eqref{cond:dct1}  due to (b). Analogous arguments show that condition \eqref{cond:dct1}  is satisfied for \linebreak $\wt\bZ^{\max}=(Z_2,\max(Z_1,Z_2))$ as well.
\end{enumerate}
\end{proof}

The following corollary is now easy to verify and provides a ready check for a model conforming to the premise of \cref{theorem:5.3}.

\begin{Corollary} \label{Corollary 2}
Let  $\bZ=\bY+\bV$ be as in Model \ref{additive:model}.  Suppose that $\E|Z_1|<\infty$  and
 $\alpha \leq  \alpha_0< 1+\alpha$.
Then the assumptions of \cref{theorem:5.3} are satisfied.
\end{Corollary}

\subsection{Asymptotic behavior of MME and MES for copula models} \label{Section 4.2}
\Cref{thm:MES} provides conditions under which we can compute the asymptotic behavior of MME and MES in an \emph{additive model} which possesses
\emph{hidden regular variation}. A question to ask here is whether a similar result would hold for heavy-tailed multivariate distributions with dependence governed by certain copulas or survival copulas exhibiting asymptotic tail independence. It turns out that the answer is positive and we can provide a suitable generalization of \Cref{thm:MES} without necessarily assuming either HRV or a tail behavior  for the distribution function $\ov F_2$ of $Z_2$.  The outcomes for MME and 
MES require mildly different conditions and hence, are stated separately.

\subsubsection{Asymptotic behavior of MME for copula models}

\begin{Theorem} \label{Theorem MME copula}
Let $\bZ\in \left[0,\infty\right)^2$ with continuous margins $F_1, F_2$, survival copula $\wh C$, $\E|Z_1|<\infty$ and
 $\ov F_{1}\in\RV_{-\alpha}$ for some $\alpha\geq 1$. Suppose  $\wh C$ has upper tail order pair $(\kappa,\tau)$ with $\kappa\geq \tau\geq 1$ and
 some slowly varying function $\ell$ at $0$ with $\lim_{s\downarrow 0}s^{\kappa-1}\ell(s)=0$ satisfying \eqref{eq:utpair} and \eqref{eq:Txy}. Also assume that
\begin{equation} \label{cond D}
    \lim_{M\to\infty}\lim_{s\downarrow 0}\int_M^{\infty}\frac{\wh C(x^{-\alpha} s ,s^\tau)}{\wh C(s,s^\tau)}\, \mathrm dx=0 \tag{\text{\rm  D}}
\end{equation}
holds.
Moreover,  suppose that $\lim_{t\to\infty}\ov F_2(t)/\ov F_1^\tau(t)=\eta\in(0,\infty)$.
Then there exists a function $a(t)\in\RV_{\frac{\kappa\alpha-\tau\alpha-1}{\tau\alpha}}$ and a constant $K\in(0,\infty)$ such that
\begin{align*}
\lim\limits_{p\downarrow 0} a(1/p) \MME(p)  = K.
\end{align*}
\end{Theorem}
%
%
\begin{proof} Due to \Cref{Lemma 2.9} we have that
$\bZ\in \MRV(\alpha, b,\nu, \E)\cap \HRV(\alpha_0, b_0,\nu_0)$ with $\alpha_0=\alpha\kappa$. The only part we need to show here is that condition
\eqref{cond D} implies condition \eqref{C1} of \Cref{thm:MES}(a). Then the  stated result is a consequence of \Cref{thm:MES}(a).

\noindent \emph{\underline{Proof of \eqref{cond D} implies  \eqref{C1}:}}
 We need to show that
  \begin{equation}\label{eq:needtoshow}
          \lim_{M\to\infty}\lim_{t\to\infty}\int_{M}^\infty \frac{\Pr(Z_1>xt, Z_2>t)}{\Pr(Z_1>t, Z_2>t)}\, \mathrm dx =   \lim_{M\to\infty}\lim_{t\to\infty}\int_{M}^\infty \frac{\wh C(\ov F_1(xt),\ov F_2(t))}{\wh C(\ov F_1(t),\ov F_2(t))} = 0.
        \end{equation}
For  notational ease, without loss of generality  we assume $\eta=1$. Let $0<\epsilon<1$. Using Potter's bound \linebreak (see \citet[Proposition 2.6(ii)]{Resnick:2007})
there exists a $t_0>0$
such that for $t\geq t_0,x\geq 1$, we have
\begin{eqnarray} \label{C4}
    (1-\epsilon)x^{-\alpha}\leq \frac{\ov F_1(xt)}{\ov F_1(t)}\leq (1+\epsilon)x^{-\alpha}
    \quad \mbox{ and }\quad (1-\epsilon)^{\tau}\leq\frac{\ov F_2(t)}{\ov F_1^{\tau}(t)}\leq (1+\epsilon)^{\tau}.
\end{eqnarray}
Hence,
\begin{eqnarray} \label{C2}
    \lefteqn{\limsup_{t\to\infty}\int_M^{\infty}\frac{\wh C(\ov F_1(xt),\ov F_2(t))}{\wh C(\ov F_1(t),\ov F_2(t))}\, \mathrm dx} \nonumber\\
        &&\leq \limsup_{t\to\infty}\int_M^{\infty}\frac{\wh C((1+\epsilon)x^{-\alpha}\ov F_1(t),(1+\epsilon)^\tau \ov F_1^\tau(t))}{\wh C(\ov F_1(t),(1-\epsilon)^\tau\ov F_1^\tau(t))}\, \mathrm dx \nonumber\\
        &&= \limsup_{s\downarrow 0}\int_M^{\infty}\frac{\wh C((1+\epsilon)x^{-\alpha}s,(1+\epsilon)^\tau s^\tau)}{\wh C(s,(1-\epsilon)^\tau s^\tau)}\, \mathrm dx \nonumber\\
        &&\leq \limsup_{\wt s\downarrow 0}\frac{C(\wt s,\wt s^\tau)}{\wh C((1+\epsilon)^{-1}\wt s,(1+\epsilon)^{-\tau}(1-\epsilon)^\tau \wt s^\tau)}\int_{M}^{\infty}\frac{\wh C(x^{-\alpha}\wt s, \wt s^\tau)}{C(\wt s,\wt s^\tau)}\, \mathrm dx\nonumber\\
        &&\leq \frac{T(1,1)}{T((1+\epsilon)^{-1},(1+\epsilon)^{-\tau}(1-\epsilon)^{\tau})}\limsup_{\wt s\downarrow 0}\int_{M}^{\infty}\frac{\wh C(x^{-\alpha}\wt s, \wt s^\tau)}{C(\wt s,\wt s^\tau)}\, \mathrm dx.
\end{eqnarray}
Similarly, we have
\begin{eqnarray} \label{C3}
    \liminf_{t\to\infty}\int_M^{\infty}\frac{\wh C(\ov F_1(xt),\ov F_2(t))}{\wh C(\ov F_1(t),\ov F_2(t))}\, \mathrm dx 
        \geq \frac{T(1,1)}{T((1-\epsilon)^{-1},(1-\epsilon)^{-\tau}(1+\epsilon)^{\tau})}\liminf_{s\downarrow 0}\int_{M}^{\infty}\frac{\wh C(x^{-\alpha}s, s^\tau)}{C(s,s^\tau)}\, \mathrm dx.
\end{eqnarray}
Since $T$ is strictly positive, using \eqref{cond D} along with \eqref{C2} and \eqref{C3}, we can conclude that \eqref{eq:needtoshow} holds. In fact, we can show in a similar fashion that \eqref{C1} (or \eqref{eq:needtoshow})  implies \eqref{cond D}, too.
\end{proof}


\begin{Remark}
Let $\bZ\in \left[0,\infty\right)^2$ with continuous margins $F_1, F_2$, $\E|Z_1|<\infty$ and
 $\ov F_{1}\in\RV_{-\alpha}$ for some $\alpha\geq 1$. Suppose that the survival copula $\wh C$ of $\bZ$
is either a Gaussian copula, a Marshall-Olkin copula or a Morgenstern copula as given in \Cref{example:cop}.
Then the assumptions of \Cref{Theorem MME copula} hold.
\end{Remark}

\subsubsection{Asymptotic behavior of  MES for copula models}

The next result complements as well as generalizes the results of \citet{Hua:Joe:2014b} where
 the asymptotic behavior of the MES was investigated for special copula families.

\begin{Theorem} \label{Theorem MES copula}
Let $\bZ\in \left[0,\infty\right)^2$ with continuous margins $F_1, F_2$, survival copula $\wh C$, $\E|Z_1|<\infty$ and
 $\ov F_{1}\in\RV_{-\alpha}$ for some $\alpha\geq 1$. Suppose  $\wh C$ has upper tail order pair $(\kappa,\tau)$ with $\kappa\geq \tau\geq 1$ and some slowly
 varying function $\ell$ at $0$ with $\lim_{s\downarrow 0}s^{\kappa-1}\ell(s)=0$ satisfying \eqref{eq:utpair} and \eqref{eq:Txy}.
Moreover, assume that for some continuous distribution function $F_2^*$ with $\lim_{t\to\infty}\ov F_2^*(t)/\ov F_1^\tau(t)= \eta\in(0,\infty)$ the asymptotic behavior
\begin{equation} \label{cond E}
    \lim_{M\to\infty}\lim_{t\to\infty}\left[\int_0^{1/M}+\int_M^\infty\right]\frac{\wh C(\ov F_1(xt),\ov F_2^*(t))}{\wh C(\ov F_1(t),\ov F_2^*(t))}\, \mathrm dx=0 \tag{\text{\rm  E}}
\end{equation}
holds. Then $(\kappa-\tau)\alpha<1$ and there exists a function $a(t)\in\RV_{\frac{\kappa\alpha-\tau\alpha-1}{\tau\alpha}}$ and a constant $K\in(0,\infty)$ such that
\begin{align*}
\lim\limits_{p\downarrow 0} a(1/p) \MES(p)  = K.
\end{align*}
\end{Theorem}
\begin{proof}
We can assume w.l.o.g that the tail of $Z_2$ is $\ov F_2^*$ (otherwise
apply the monotone transformation $F_2^{*\leftarrow}\circ F_2$ on $Z_2$ which does not
change the MES and the copula). If the tail of $Z_2$ is $\ov F_2^*$ then the equivalence of \eqref{cond E} and \eqref{cond:dct1} is easy to check. Thus,
the conclusion for the asymptotic behavior of MES follows
from \Cref{Lemma 2.9} and \Cref{thm:MES}(b). Finally, \Cref{Remark 4.2}(2) and
$\ov F_2(t)\sim \eta\ov F_1^\tau(t)\in\RV_{-\alpha\tau}$ implies $(\kappa-\tau)\alpha<1$.
\end{proof}

\begin{Corollary} \label{Theorem MES copula Corollary}
Let $\bZ=(Z_1,Z_2)\in \left[0,\infty\right)^2$ with  survival copula $\wh C$, continuous margins $F_1,F_2$, $\E|Z_1|<\infty$ and
 $\ov F_{1}(t)\sim K_1t^{-\alpha}$ for some $\alpha> 1$ and  constant $K_1\in(0,\infty)$. Suppose  $\wh C$ has upper tail order pair $(\kappa,\tau)$ with
 $\kappa\geq \tau\geq 1$ and some slowly varying function $\ell$ at $0$ with $\lim_{s\downarrow 0}s^{\kappa-1}\ell(s)=0$ satisfying \eqref{eq:utpair} and \eqref{eq:Txy}.
Moreover,
\begin{equation} \label{cond F}
    \lim_{M\to\infty}\lim_{s\downarrow 0}\left[\int_0^M+\int_M^{\infty}\right]\frac{\wh C(x^{-\alpha} s ,s^\tau)}{\wh C(s,s^\tau)}\, \mathrm dx=0 \tag{\text{\rm  F}}
\end{equation}
holds. Then $(\kappa-\tau)\alpha<1$ and there exists a function $a(t)\in\RV_{\frac{\kappa\alpha-\tau\alpha-1}{\tau\alpha}}$  and a constant $K\in(0,\infty)$ such that
\begin{align*}
\lim\limits_{p\downarrow 0} a(1/p) \MES(p)  = K.
\end{align*}
\end{Corollary}

\begin{proof}
Similar to the proof of Theorem \ref{Theorem MME copula}, the proof here follows easily if we show that conditions \eqref{cond E} and \eqref{cond F} are
equivalent. However, since Potter's bounds hold only for $x\geq 1$
we require the additional assumption that the slowly varying part in the tail of $F_1$ behaves like a constant to obtain a similar bound as \eqref{C4} for
$0<x\leq 1$. Then the result is a direct consequence of \Cref{Theorem MES copula}.
\end{proof}
\begin{Remark} A few observations from the above results are noted below.
\begin{enumerate}[(1)]
\item The result shows that the asymptotic behavior of  the MES is determined only by the dependence structure and the tail behavior of $Z_1$; the tail behavior of
$Z_2$ has no influence. 
Particularly, we see that HRV is not a necessary assumption.
\item An analogous result for the MME does not hold, since a monotone transformation of $Z_2$
will in fact change the MME in contrast to the MES; the tail of $Z_2$ has an influence on the limit behavior of MME.
Further, note that \eqref{cond D} is only an assumption on the upper tail dependence in contrast to \eqref{cond F} where the whole dependence structure plays  a role as well.

\item A result similar to \Cref{Theorem MES copula Corollary} under stronger assumptions has been discussed in \citet[Proposition 2.1]{CaiMusta2017}.
  Inter alia they assume the slowly varying function $\ell$ to be a constant,
 $x\mapsto T(x,1)$ to be continuous and $\tau=1$.

\item The copula examples in  \Cref{example:cop}
only satisfy \eqref{cond D} but not \eqref{cond F} and hence, \Cref{Theorem MES copula Corollary} cannot be applied.
However, such examples are covered in \citet[Section 3.4]{Hua:Joe:2014b} for either
Pareto or Weibull-margins. In these examples the rate of increase of the MES is slower than in the
asymptotic tail dependent case but faster than under condition \eqref{cond F}.
\end{enumerate}
\end{Remark}

The rest of this section is dedicated to construct examples of survival copulas that satisfy the assumptions
 of \Cref{Theorem MES copula}. The examples are created using the additive structure in Model \ref{additive:model} and Bernoulli mixture models as
 discussed in \citet[Section 5]{Hua:Joe:Li:2014} and \citet[Example 2]{DasFasen2017}. First, we propose a result which we apply on the suggested models. Note that the models
 in the examples are not created using copulas apriori but we use the inherent copula structure governing the generation method in order to obtain the examples.

\begin{Proposition} \label{Prop 4.11}
Let   $\bZ \in \MRV({\alpha}, b, \nu, \E)\cap \HRV({\alpha_0}, b_0, \nu_0, \E_0)$   with continuous margins $F_1$, $F_2$, $\E|Z_1|<\infty$, $\lim_{t\to\infty}\ov F_1(t)/\ov F_2^\tau(t)=1 $
for some $\alpha_0/\alpha\geq\tau\geq 1$ and suppose \eqref{cond:dct1} holds. Denote by $$\widehat C(u,v)=u+v-1+F_{\bZ}(F_1^{\leftarrow}(1-u),F_2^{\leftarrow}(1-v))$$ the
survival copula of $\bZ$. Furthermore, let $\bZ^*=(Z_1^*,Z_2^*)\in[0,\infty)^2$ be a random vector with survival copula
  $\widehat C$ and marginal distribution function $F_1$ of $Z_1^*$ and some continuous distribution function $ F_2^*$ of $Z_2^*$. Then with \linebreak
  $a(t)=b_0^{\leftarrow}\{\VaR_{1-1/t}(Z_2)\}/\{t\VaR_{1-1/t}(Z_2)\}\in\RV_{(\alpha_0-\alpha\tau-1)/\alpha\tau}$ we have
\begin{align*}
\lim\limits_{p\downarrow 0} a(1/p) \E(Z_1^*|Z_2^*>\VaR_{1-p}(Z_2^*))   = \int_0^{\infty} \nu_0((x,\infty)\times(1,\infty))\;\mathrm dx
       \end{align*}
where $0<\int_0^{\infty}\nu_0((x,\infty)\times(1,\infty))\;\mathrm dx<\infty$.
\end{Proposition}
\begin{proof}
Using \Cref{Lemma:Hua:Joe:Li:1} the upper tail order function of $\widehat C$ exists with upper tail order pair $(\kappa,\tau)=(\alpha_0/\alpha,\tau)$, i.e.,
$\wh C(s,s^\tau)\sim s^{\kappa}\ell(s)$ as $s\downarrow 0$.
Further, $\lim_{s\downarrow 0}s^{\kappa-1}\ell(s)=0$ due to \Cref{Lemma 2.5}  and \Cref{Lemma 3.6}. Moreover, \eqref{cond:dct1} proved in \Cref{theorem:5.3} implies \eqref{cond E}.
Hence, the result is a consequence of   \Cref{Theorem MES copula}.
\end{proof}

\begin{Example}
Let  $\bZ=\bY+\bV$ be as in Model \ref{additive:model} with continuous margins for $Y_{1},Y_{2},V_{1}, V_{2}$ and suppose $\ov F_{Y_2}\in\RV_{-\alpha^*}$ with
  $\alpha+\alpha^*>\alpha_0$. Further, let $\bZ^*=(\bZ^*_1,\bZ^*_2)$ be defined as in \Cref{Prop 4.11}.
  Then there exists a function $a(t)\in\RV_{(\alpha_0-\alpha^*-1)/\alpha^*}$  and a constant $K\in(0,\infty)$ such that
\begin{align*}
\lim\limits_{p\downarrow 0} a(1/p) \E(Z_1^*|Z_2^*>\VaR_{1-p}(Z_2^*))   = K.
\end{align*}
\begin{proof}
The conclusion is easy to see since by \Cref{theorem:3.1} and \Cref{theorem:5.3} we already know that  \linebreak $\bZ \in \MRV({\alpha}, b, \nu, \E)\cap \HRV({\alpha_0}, b_0, \nu_0, \E_0)$
and \eqref{cond:dct1} holds. The rest is a consequence of \Cref{Prop 4.11}.
\end{proof}

 Clearly, analogous results hold if $\widehat C$ is the copula of the other examples of vectors defined in \Cref{theorem:5.3}.

\end{Example}

\begin{Example}
This model is motivated by the Bernoulli mixture type models discussed in \citet[Section 5]{Hua:Joe:Li:2014} and \citet[Example 2]{DasFasen2017}.
Suppose that $X_1,X_2,X_3$ are independent  Pareto random variables with parameters $\alpha$, $\alpha_0$ and $\gamma$, respectively, where $1<\alpha<\alpha_0<\gamma$,
$\alpha+\gamma>\alpha_0$.
Let $B$ be a Bernoulli$(q)$ random variable  with $q \in (0,1)$, $\bR=(R_1,R_2)$ be a random vector with each margin defined on $[1,\infty)$  and
$\E\|\bR\|^{\alpha_0}<\infty$. We also assume  $X_1,X_2,X_3,B,\bR$ are independent of each other. Now define
$$\bZ=(Z_1, Z_2) = B(X_1,X_3) + (1-B)(R_1X_2,R_2X_2)$$ 
and let $\bZ^*=(\bZ^*_1,\bZ^*_2)$ be defined as in \Cref{Prop 4.11}.
Then
\begin{align*}
\lim\limits_{p\downarrow 0} p^{\frac{1}{\alpha_0}} \E(Z_1^*|Z_2^*>\VaR_{1-p}(Z_2^*))   = (1-q)\int_0^\infty\E(\min(x^{-1}R_1,R_2))^{\alpha_0}\, \mathrm dx.
       \end{align*}
\end{Example}
\begin{proof}
We only need to verify that $\bZ \in \MRV({\alpha}, b, \nu, \E)\cap \HRV({\alpha_0}, b_0, \nu_0, \E_0)$
and that \eqref{cond:dct1} is satisfied since the rest is a consequence of \Cref{Prop 4.11}. However, it is easy to check (cf. the similar models
 in \cite[Section 5]{Hua:Joe:Li:2014} and \cite[Example 2]{DasFasen2017})
that $\bZ \in \MRV({\alpha}, b, \nu, \E)\cap \HRV({\alpha_0}, b_0, \nu_0, \E_0)$ with $b(t)=t^{1/\alpha}$, $\nu(dx ,dy)=q\alpha x^{-\alpha-1}\, \mathrm dx\times \varepsilon_0(dy)$,
$b_0(t)=t^{1/\alpha_0}$
and $\nu_0((x,\infty)\times(y,\infty))=(1-q)\E(\min(x^{-1}R_1,y^{-1}R_2)^{\alpha_0})$ for $x,y>0$. Moreover, for $t,x\geq 1$, we have the inequality
\begin{eqnarray*}
    (1-q)t^{-\alpha_0}x^{-\alpha_0}\E(\min(R_1,R_2)^{\alpha_0})\leq  \Pr(Z_1>xt,Z_2>t)\leq (q+(1-q)\E(R_1^{\alpha_0}))t^{-\alpha_0}x^{-\alpha}
\end{eqnarray*}
and for $0<x\leq 1$,
\begin{eqnarray*}
    \Pr(Z_1>xt,Z_2>t)\leq \Pr(Z_2>t)\leq (q+(1-q)\E(R_2^{\alpha_0}))t^{-\alpha_0}.
\end{eqnarray*}
Thus, condition \eqref{cond:dct1} is also satisfied.
\end{proof}

\section{Conclusion}\label{sec:concl}
Our goal in this paper was to investigate certain conditional excess measures for bivariate models with asymptotic tail independence and exhibiting heavy tails
in the margins. We have been able to find asymptotic rates of convergence for the measures MES, MME as well as $\MES^{+}, \MES^{\min}, \MES^{\max}$ for a variety of
copula models, additive models and Bernoulli mixture models. We particularly note that the limit behavior of MES only depends on the tail of the survival copula and
the tail behavior of the variable which is not-conditioned (denoted by $Z_{1}$ in most of our examples). The asymptotic behavior of MME involves further information
on the copula as well as the tail of the conditioning variable ($Z_{2}$ in our examples). In addition we constructed a large class of \emph{hidden regularly varying}
 models useful in the  context of systemic risks which were not known or used hitherto up to our knowledge.
Interesting extensions of our results to multivariate structures beyond $d=2$
(see \citet{hoffmann_etal:2016,hoffmann:2017}) as well as graphical and network structures (see \citet{kley:kluppelberg:reinert:2016,kley:kluppelberg:reinert:2017})
are possible and are topics of future research. 


\section*{References}
\bibliographystyle{plainnat}
\bibliography{bibeshrv}

\begin{thebibliography}{53}
\providecommand{\natexlab}[1]{#1}
\providecommand{\url}[1]{\texttt{#1}}
\expandafter\ifx\csname urlstyle\endcsname\relax
  \providecommand{\doi}[1]{doi: #1}\else
  \providecommand{\doi}{doi: \begingroup \urlstyle{rm}\Url}\fi

\bibitem[Acharya et~al.(2010)Acharya, Pedersen, Philippon, and
  Richardson]{acharya:petersen:philippon:richardson:2010}
V.~V. Acharya, L.H. Pedersen, T.~Philippon, and M.P. Richardson.
\newblock Measuring systemic risk.
\newblock \emph{AFA 2011 Denver Meetings Paper}, May 2010.
\newblock \url{http://ssrn.com/abstract=1573171}.

\bibitem[Anderson and Meerschaert(1998)]{anderson:meerschaert:1998}
P.L. Anderson and M.M. Meerschaert.
\newblock Modeling river flows with heavy tails.
\newblock \emph{Water Resour. Res.}, 34\penalty0 (9):\penalty0 2271--2280,
  1998.

\bibitem[Ballerini(1994)]{Ballerini}
R.~Ballerini.
\newblock Archimedean copulas, exchangeability, and max-stability.
\newblock \emph{J. Appl. Probab.}, 31\penalty0 (2):\penalty0 383--390, 1994.

\bibitem[Barg\`es et~al.(2009)Barg\`es, Cossette, and
  Marceau]{Barges:Cossette:Marceau}
M.~Barg\`es, H.~Cossette, and E.~Marceau.
\newblock T{V}a{R}-based capital allocation with copulas.
\newblock \emph{Insurance Math. Econom.}, 45\penalty0 (3):\penalty0 348--361,
  2009.

\bibitem[Bingham et~al.(1989)Bingham, Goldie, and
  Teugels]{bingham:goldie:teugels:1989}
N.~H. Bingham, C.~M. Goldie, and J.~L. Teugels.
\newblock \emph{Regular Variation}.
\newblock Cambridge University Press, Cambridge, 1989.

\bibitem[Brownlees and Engle(2015)]{brownlees:engle:2017}
C.~Brownlees and R.~Engle.
\newblock Srisk: A conditional capital shortfall index for systemic risk
  management.
\newblock \emph{Rev. Financial Stud.}, 30\penalty0 (1):\penalty0 48--79, 2015.

\bibitem[Cai and Li(2005)]{Cai:Li:2005}
J.~Cai and H.~Li.
\newblock Conditional tail expectations for multivariate phase-type
  distributions.
\newblock \emph{J. Appl. Probab.}, 42\penalty0 (3):\penalty0 810--825, 2005.

\bibitem[Cai and Musta(2017)]{CaiMusta2017}
J.-J. Cai and E.~Musta.
\newblock Estimation of the marginal expected shortfall under asymptotic
  independence.
\newblock \emph{Preprint}, 2017.
\newblock \url{https://arxiv.org/abs/1709.04285}.

\bibitem[Cai et~al.(2015)Cai, Einmahl, de~Haan, and
  Zhou]{cai:einmahl:dehaan:zhou:2015}
J.-J. Cai, J.H.J. Einmahl, L.~de~Haan, and C.~Zhou.
\newblock Estimation of the marginal expected shortfall: the mean when a
  related variable is extreme.
\newblock \emph{J. Roy. Statist. Soc. Ser. B}, 77\penalty0 (2):\penalty0
  417--442, 2015.

\bibitem[Cap{\'e}ra{\`a} et~al.(2000)Cap{\'e}ra{\`a}, Foug{\`e}res, and
  Genest]{Capera:Fougeres:Genest}
P.~Cap{\'e}ra{\`a}, A.-L. Foug{\`e}res, and C.~Genest.
\newblock Bivariate distributions with given extreme value attractor.
\newblock \emph{J. Multivariate Anal.}, 72\penalty0 (1):\penalty0 30--49, 2000.

\bibitem[Charpentier and Segers(2009)]{Charpentier:Segers}
A.~Charpentier and J.~Segers.
\newblock Tails of multivariate {A}rchimedean copulas.
\newblock \emph{J. Multivariate Anal.}, 100\penalty0 (7):\penalty0 1521--1537,
  2009.

\bibitem[Chiragiev and Landsman(2007)]{Chiragiev:Landsman:2007}
A.~Chiragiev and Z.~Landsman.
\newblock Multivariate {P}areto portfolios: {TCE}-based capital allocation and
  divided differences.
\newblock \emph{Scand. Actuar. J.}, 2007\penalty0 (4):\penalty0 261--280, 2007.

\bibitem[Cousin and Di~Bernardino(2014)]{Cousin:Bernardino:2013}
A.~Cousin and E.~Di~Bernardino.
\newblock On multivariate extensions of conditional-tail-expectation.
\newblock \emph{Insurance Math. Econom.}, 55\penalty0 (C):\penalty0 272--282,
  2014.

\bibitem[Crovella et~al.(1999)Crovella, Bestavros, and
  Taqqu]{crovella:bestavros:taqqu:1998}
M.~Crovella, A.~Bestavros, and M.S. Taqqu.
\newblock Heavy-tailed probability distributions in the world wide web.
\newblock In M.S.~Taqqu R.~Adler, R.~Feldman, editor, \emph{A Practical Guide
  to Heavy Tails: Statistical Techniques for Analysing Heavy Tailed
  Distributions}. Birkh{\"a}user, Boston, 1999.

\bibitem[Das and Fasen-Hartmann(2018)]{DasFasen2017}
B.~Das and V.~Fasen-Hartmann.
\newblock Risk contagion under regular variation and asymptotic tail
  independence.
\newblock \emph{J. Multivariate Anal.}, 165:\penalty0 194--215, 2018.

\bibitem[Das and Resnick(2015)]{das:resnick:2015}
B.~Das and S.I. Resnick.
\newblock Models with hidden regular variation: generation and detection.
\newblock \emph{Stochastic Systems}, 5\penalty0 (2):\penalty0 195--238, 2015.

\bibitem[Das et~al.(2013)Das, Mitra, and Resnick]{das:mitra:resnick:2013}
B.~Das, A.~Mitra, and S.I. Resnick.
\newblock Living on the multidimensional edge: seeking hidden risks using
  regular variation.
\newblock \emph{Adv. in Appl. Probab.}, 45\penalty0 (1):\penalty0 139--163,
  2013.

\bibitem[Embrechts et~al.(1997)Embrechts, Kl\"{u}ppelberg, and
  Mikosch]{embrechts:kluppelberg:mikosch:1997}
P.~Embrechts, C.~Kl\"{u}ppelberg, and T.~Mikosch.
\newblock \emph{Modelling Extreme Events for Insurance and Finance}.
\newblock Springer-Verlag, Berlin, 1997.

\bibitem[Hoffmann(2017)]{hoffmann:2017}
H.~Hoffmann.
\newblock Multivariate conditional risk measures.
\newblock PhD Thesis, 2017.

\bibitem[Hoffmann et~al.(2016)Hoffmann, Meyer-Brandis, and
  Svindland]{hoffmann_etal:2016}
H.~Hoffmann, T.~Meyer-Brandis, and G.~Svindland.
\newblock Risk-consistent conditional systemic risk measures.
\newblock \emph{Stochastic Process. Appl.}, 126\penalty0 (7):\penalty0
  2014--2037, 2016.

\bibitem[Hua and Joe(2011{\natexlab{a}})]{Hua:Joe:2011}
L.~Hua and H.~Joe.
\newblock Tail order and intermediate tail dependence of multivariate copulas.
\newblock \emph{J. Multivariate Anal.}, 102\penalty0 (10):\penalty0 1454--1471,
  2011{\natexlab{a}}.

\bibitem[Hua and Joe(2011{\natexlab{b}})]{Hua:Joe:2011b}
L.~Hua and H.~Joe.
\newblock Second order regular variation and conditional tail expectation of
  multiple risks.
\newblock \emph{Insurance Math. Econom.}, 49\penalty0 (3):\penalty0 537--546,
  2011{\natexlab{b}}.

\bibitem[Hua and Joe(2013)]{Hua:Joe:2013b}
L.~Hua and H.~Joe.
\newblock Intermediate tail dependence: a review and some new results.
\newblock In \emph{Stochastic orders in reliability and risk}, volume 208,
  pages 291--311. Springer-Verlag, New York, 2013.

\bibitem[Hua and Joe(2014)]{Hua:Joe:2014b}
L.~Hua and H.~Joe.
\newblock Strength of tail dependence based on conditional tail expectation.
\newblock \emph{J. Multivariate Anal.}, 123:\penalty0 143--159, 2014.

\bibitem[Hua et~al.(2014)Hua, Joe, and Li]{Hua:Joe:Li:2014}
L.~Hua, H.~Joe, and H.~Li.
\newblock Relations between hidden regular variation and the tail order of
  copulas.
\newblock \emph{J. Appl. Probab.}, 51\penalty0 (1):\penalty0 37--57, 2014.

\bibitem[Hult and Lindskog(2006)]{hult:lindskog:2006a}
H.~Hult and F.~Lindskog.
\newblock Regular variation for measures on metric spaces.
\newblock \emph{Publ. Inst. Math. (Beograd) (N.S.)}, 80(94):\penalty0 121--140,
  2006.

\bibitem[Jaworski(2006)]{Jaworski2006}
P.~Jaworski.
\newblock On uniform tail expansions of multivariate copulas and wide
  convergence of measures.
\newblock \emph{Appl. Math. (Warsaw)}, 33\penalty0 (2):\penalty0 159--184,
  2006.

\bibitem[Jessen and Mikosch(2006)]{jessen:mikosch}
A.H. Jessen and T.~Mikosch.
\newblock Regularly varying functions.
\newblock \emph{Publ. Inst. Math. (Beograd) (N.S.)}, 80(94):\penalty0 171--192,
  2006.

\bibitem[Joe and Li(2011)]{Joe:Li:2011}
H.~Joe and H.~Li.
\newblock Tail risk of multivariate regular variation.
\newblock \emph{Methodol. Comput. Appl. Probab.}, 13\penalty0 (4):\penalty0
  671--693, 2011.

\bibitem[Kley et~al.(2016)Kley, Kl{\"u}ppelberg, and
  Reinert]{kley:kluppelberg:reinert:2016}
O.~Kley, C.~Kl{\"u}ppelberg, and G.~Reinert.
\newblock Risk in a large claims insurance market with bipartite graph
  structure.
\newblock \emph{Operations Research}, 64\penalty0 (5):\penalty0 1159--1176,
  2016.

\bibitem[Kley et~al.(2017)Kley, Kl\"uppelberg, and
  Reinert]{kley:kluppelberg:reinert:2017}
O.~Kley, C.~Kl\"uppelberg, and G.~Reinert.
\newblock Conditional risk measures in a bipartite market structure.
\newblock \emph{Scand. Actuar. J. (forthcoming)}, 2017.
\newblock \url{https://arxiv.org/abs/1510.00616}.

\bibitem[Kulik and Soulier(2015)]{KulikSoulier2015}
R.~Kulik and P.~Soulier.
\newblock Heavy tailed time series with extremal independence.
\newblock \emph{Extremes}, 18\penalty0 (2):\penalty0 273--299, 2015.

\bibitem[Landsman and Valdez(2003)]{Landsman:Valdez}
Z.~Landsman and E.~Valdez.
\newblock Tail conditional expectations for elliptical distributions.
\newblock \emph{N. Am. Actuar. J.}, 7\penalty0 (4):\penalty0 55--71, 2003.

\bibitem[Ledford and Tawn(1997)]{Ledford:Tawn}
A.W. Ledford and J.A. Tawn.
\newblock Modelling dependence within joint tail regions.
\newblock \emph{J. Roy. Statist. Soc. Ser. B}, 59\penalty0 (2):\penalty0
  475--499, 1997.

\bibitem[Li(2017)]{li:2016}
H.~Li.
\newblock Operator tail dependence of copulas.
\newblock \emph{Methodol. Comput. Appl. Probab.}, 2017:\penalty0 1--15, 2017.

\bibitem[Li and Hua(2015)]{Li:Hua:2015}
H.~Li and L.~Hua.
\newblock Higher order tail densities of copulas and hidden regular variation.
\newblock \emph{J. Multivariate Anal.}, 138:\penalty0 143--155, 2015.

\bibitem[Lindskog et~al.(2014)Lindskog, Resnick, and
  Roy]{lindskog:resnick:roy:2014}
F.~Lindskog, S.I. Resnick, and J.~Roy.
\newblock Regularly varying measures on metric spaces: hidden regular variation
  and hidden jumps.
\newblock \emph{Probab. Surveys}, 11:\penalty0 270--314, 2014.

\bibitem[Maulik and Resnick(2005)]{maulik:resnick:2005}
K.~Maulik and S.I. Resnick.
\newblock Characterizations and examples of hidden regular variation.
\newblock \emph{Extremes}, 7\penalty0 (1):\penalty0 31--67, 2005.

\bibitem[McNeil et~al.(2005)McNeil, Frey, and Embrechts]{McNeil:Frey:Embrechts}
A.J. McNeil, R.~Frey, and P.~Embrechts.
\newblock \emph{Quantitative Risk Management}.
\newblock Princeton University Press, Princeton, 2005.

\bibitem[Nelsen(2006)]{Nelsen}
R.B. Nelsen.
\newblock \emph{An Introduction to Copulas}.
\newblock Springer Series in Statistics. Springer-Verlag, New York, second
  edition, 2006.

\bibitem[Poon et~al.(2004)Poon, Rockinger, and Tawn]{Poon:Rockinger:Tawn}
S.-H. Poon, M.~Rockinger, and J.~Tawn.
\newblock Extreme value dependence in financial markets: Diagnostics, models,
  and financial implications.
\newblock \emph{Rev. Financial Stud.}, 17\penalty0 (2):\penalty0 581--610,
  2004.

\bibitem[Reiss(1989)]{reiss:1989}
R.-D. Reiss.
\newblock \emph{Approximate Distributions of Order Statistics}.
\newblock Springer-Verlag, New York, 1989.

\bibitem[Resnick(2002)]{resnick:2002}
S.I. Resnick.
\newblock Hidden regular variation, second order regular variation and
  asymptotic independence.
\newblock \emph{Extremes}, 5\penalty0 (4):\penalty0 303--336, 2002.

\bibitem[Resnick(2007)]{Resnick:2007}
S.I. Resnick.
\newblock \emph{Heavy Tail Phenomena: Probabilistic and Statistical Modeling}.
\newblock Springer-Verlag, New York, 2007.

\bibitem[Resnick(2008)]{resnickbook:2008}
S.I. Resnick.
\newblock \emph{Extreme Values, Regular Variation and Point Processes}.
\newblock Springer Series in Operations Research and Financial Engineering.
  Springer, New York, 2008.
\newblock Reprint of the 1987 original.

\bibitem[Salmon(2009)]{salmon:2009}
F.~Salmon.
\newblock Recipe for disaster: the formula that killed {W}all {S}treet.
\newblock February 23, Wired Magazine, 2009.

\bibitem[Sibuya(1960)]{sibuya:1960}
M.~Sibuya.
\newblock Bivariate extreme statistics.
\newblock \emph{Ann. Inst. Stat. Math.}, 11:\penalty0 195--210, 1960.

\bibitem[Smith(2003)]{smith:2003}
R.L. Smith.
\newblock Statistics of extremes, with applications in environment, insurance
  and finance.
\newblock In B.~Finkenstadt and H.~Rootz\'en, editors, \emph{Extreme Values in
  Finance, Telecommunications, and the Environment}, pages 1--78. Chapman-Hall,
  London, 2003.

\bibitem[Tankov(2016)]{tankov:2016a}
P.~Tankov.
\newblock Tails of weakly dependent random vectors.
\newblock \emph{J. Multivariate Anal.}, 145:\penalty0 73--86, 2016.

\bibitem[Wadsworth and Tawn(2013)]{wadsworth:tawn:2013}
J.L. Wadsworth and J.A. Tawn.
\newblock A new representation for multivariate tail probabilities.
\newblock \emph{Bernoulli}, 19\penalty0 (5B):\penalty0 2689--2714, 2013.

\bibitem[Weller and Cooley(2014)]{weller:cooley:2014}
G.B. Weller and D.~Cooley.
\newblock A sum characterization of hidden regular variation with likelihood
  inference via expectation-maximization.
\newblock \emph{Biometrika}, 101\penalty0 (1):\penalty0 17--36, 2014.

\bibitem[Zhou(2010)]{zhou:2010}
C.~Zhou.
\newblock Are banks too big to fail? measuring systemic importance of financial
  institutions.
\newblock \emph{Int. J. Cent. Bank.}, 2010.
\newblock URL \url{http://www.ijcb.org/journal/ijcb10q4a10.pdf}.

\bibitem[Zhu and Li(2012)]{Zhu:Li:2012}
L.~Zhu and H.~Li.
\newblock Asymptotic analysis of multivariate tail conditional expectations.
\newblock \emph{N. Am. Actuar. J.}, 16\penalty0 (3):\penalty0 350--363, 2012.

\end{thebibliography}

\end{document}